\documentclass[12pt]{article}
\usepackage{amsmath,amssymb,amsthm,bbm,natbib,color,hyperref}
\usepackage[margin=1.2in]{geometry}
\usepackage{caption}
\usepackage{graphicx}
\usepackage{subcaption}
\usepackage{afterpage}
\usepackage[hang,flushmargin]{footmisc}

\newtheorem{theorem}{Theorem}
\newtheorem{lemma}{Lemma}

\newtheorem{definition}{Definition}
\newtheorem{corollary}{Corollary}
\newcommand{\R}{\mathbb{R}}
\newcommand{\eps}{\epsilon}

\newcommand{\EE}[1]{\mathbb{E}\left[{#1}\right]}
\newcommand{\EEst}[2]{\mathbb{E}\left[{#1}\  \middle| \ {#2}\right]}
\newcommand{\Ep}[2]{\mathbb{E}_{{#1}}\left[{#2}\right]}
\newcommand{\Epst}[3]{\mathbb{E}_{{#1}}\left[{#2}\  \middle| \ {#3}\right]}
\newcommand{\PP}[1]{\mathbb{P}\left\{{#1}\right\}}
\newcommand{\PPst}[2]{\mathbb{P}\left\{{#1}\  \middle| \ {#2}\right\}}
\newcommand{\Ppst}[3]{\mathbb{P}_{{#1}}\left\{{#2}\  \middle| \ {#3}\right\}}
\newcommand{\Pp}[2]{\mathbb{P}_{{#1}}\left\{{#2}\right\}}

\newcommand{\One}[1]{{\mathbbm{1}}\left\{{#1}\right\}}
\newcommand{\one}[1]{{\mathbbm{1}}_{{#1}}}
\newcommand{\iidsim}{\stackrel{\textnormal{iid}}{\sim}}

  \newcommand\independent{\protect\mathpalette{\protect\independenT}{\perp}}
\def\independenT#1#2{\mathrel{\rlap{$#1#2$}\mkern2mu{#1#2}}}

\DeclareMathOperator*{\argmin}{argmin}

\newcommand{\Fcal}{\mathcal{F}}
\newcommand{\Xcal}{\mathcal{X}}
\newcommand{\Rcal}{\mathcal{R}}

\newcommand{\Ecal}{\mathcal{E}}
\newcommand{\Qcal}{\mathcal{Q}}
\newcommand{\Lcal}{\mathcal{L}}
\newcommand{\Ch}{\widehat{C}}
\newcommand{\leb}{\textnormal{leb}}
\newcommand{\ba}{\mathbf{a}}
\newcommand{\bt}{\mathbf{t}}
\newcommand{\bpi}{\boldsymbol{\pi}}
\newcommand{\dtv}{\textnormal{d}_{\textnormal{TV}}}

\newcommand{\edit}[1]{{{#1}}}

\title{Is distribution-free inference possible \\ for binary regression?}
\author{Rina Foygel Barber}
\date{}

\begin{document}
\maketitle

\begin{abstract}
For a regression problem with a binary label response, 
we examine the problem of constructing confidence intervals for the label probability conditional on the features.
In a setting where we do not have any information about the underlying distribution,
we would ideally like to provide confidence intervals 
that are distribution-free---that is, valid with no assumptions on the distribution of the data.
Our results establish an explicit lower bound on the length of any distribution-free confidence interval, and
 construct a procedure that can approximately achieve this length. 
In particular, this lower bound is independent of the sample size and holds for all distributions
with no point masses, meaning that it is not possible for any
 distribution-free procedure to be adaptive with respect to any type of special structure in the distribution.
\end{abstract}
\section{Introduction}
Consider a regression problem where we would like to model the relationship between a feature vector $X\in\R^d$
and a response $Y\in\R$, based on a sample of $n$ data points, $(X_1,Y_1),\dots,(X_n,Y_n)\iidsim P$.
 In a high-dimensional setting where $d$ is large, 
many modern methods are available to build powerful predictive models for $Y$ given $X$, but relatively little
is known about their theoretical properties---for example, if we train a neural network on the $n$ available data points,
can we quantify its accuracy on unseen test data, without making strong assumptions on $P$, the unknown distribution 
of the data? 

If we are willing to assume that the data follows a regression model $ \Epst{P}{Y}{X=x} = f(x)$
where the function $f$ satisfies certain assumptions, 
then classical statistical results assure that these questions can be answered using more simple regression methods.
For example, if $f(x)$ lies in a parametric family (e.g., linear regression) then we can perform inference within this parametric 
model. In a more general nonparametric setting, if $f(x)$ is assumed to satisfy some smoothness
conditions, classical nonparametric methods such as nearest neighbors will also yield guarantees on the accuracy
of our estimate of $f(x)$. 
However, the reported results will be invalid if the assumptions (the parametric model, or the 
smoothness conditions) do not hold. 

To address this concern, the recent field of {\em distribution-free prediction} considers the problem of providing
valid predictive inference without any assumptions on the data distribution.
 The aim of distribution-free prediction is formulated as follows: given a training data
set $(X_1,Y_1),\dots,(X_n,Y_n)\in\R^d\times\R$, our task is to
construct a map $\Ch_n$, mapping a new data point $x\in\R^d$
to an interval or set $\Ch_n(x)\subseteq\R$, such that
\begin{equation}\label{eqn:distr_free_prediction}\Pp{(X_i,Y_i)\iidsim P}{Y_{n+1}  \in \Ch_n(X_{n+1})} \geq 1-\alpha\textnormal{\quad for
 all distributions $P$ on $\R^d\times\R$}.\end{equation}
Here the probability is taken with respect to $(X_i,Y_i)\iidsim P$ for $i=1,\dots,n+1$ (the training and test data
are drawn from the same distribution $P$). The bound is required to hold uniformly over all distributions $P$, without
constraining to, say, distributions that satisfy some notion of smoothness.
For example, the conformal inference methodology \citep{vovk2005algorithmic} provides an elegant framework
for distribution-free prediction, and can adapt to the favorable properties of the underlying distribution
to achieve asymptotically optimal prediction intervals in certain settings (see, e.g., \citet{lei2014distribution,lei2018distribution}).

Distribution-free
prediction has also been studied in the context of a binary response $Y\in\{0,1\}$,
where the output is a set $\Ch_n(X_{n+1})\subseteq\{0,1\}$ (or more generally, in a setting with a finite
set of possible labels) \citep{vovk2005algorithmic,lei2014classification,sadinle2019least}. For a binary $Y$, the goal is to satisfy
\begin{equation}\label{eqn:distr_free_prediction_binary}\Pp{(X_i,Y_i)\iidsim P}{Y_{n+1}  \in \Ch_n(X_{n+1})} \geq 1-\alpha\textnormal{\quad for
 all distributions $P$ on $\R^d\times\{0,1\}$}.\end{equation}
For distributions $P$ where the conditional probability $\pi_P(X) = \Ppst{P}{Y=1}{X}$ is typically close to either 0 or 1,
given sufficient data the resulting distribution-free predictive set $\Ch_n(X_{n+1})$ can often be a singleton set, $\{0\}$ or $\{1\}$.
If the labels are inherently noisy, however---that is, if $\pi_P(X)$ is typically bounded away from both 0 and 1---then 
$\{0,1\}$ will often be the only possible set offering guaranteed predictive coverage, even if we were to have oracle
knowledge of the distribution $P$. In other words, predictive coverage (whether distribution-free or not)
is not a meaningful target for binary regression problems
with noisy labels; we would like to estimate the label probability $\pi_P(X)$ directly, rather than try to predict the inherently noisy label $Y$.

\subsection{Summary of contributions}
In this work, we ask whether the distribution-free framework can be extended beyond the prediction task,
in the binary regression setting. We will aim to
 provide distribution-free inference on the conditional label probability $\pi_P(X)=\Ppst{P}{Y=1}{X}$.
We are particularly interested in scenarios where $\pi_P(X)$ is typically not close to either 0 or 1 and so meaningful predictive inference 
would not be possible even if $P$ were known.
In this type of setting,
is it nonetheless possible to provide a nontrivial confidence interval for $\pi_P(X)$, and to 
ensure a distribution-free guarantee of coverage?

Specifically, our goal is to investigate the feasibility of constructing an algorithm that satisfies the following condition:
\begin{definition}\label{def:distr_free_confidence_binary}
An algorithm $\Ch_n$ provides a $(1-\alpha)$-distribution-free confidence interval for binary regression if it holds that
\begin{equation}\label{eqn:distr_free_confidence_binary}\Pp{(X_i,Y_i)\iidsim P}{\pi_P(X_{n+1})  \in \Ch_n(X_{n+1})} \geq 1-\alpha\textnormal{\ for
 all distributions $P$ on $\R^d\!\times\! \{0,1\}$}.\end{equation}
\end{definition}
\noindent This notion of a valid distribution-free confidence interval was previously studied by \citet[Section 5.2]{vovk2005algorithmic}, under
the name ``weakly valid probability estimators''.

Definition~\ref{def:distr_free_confidence_binary} requires a fairly weak form of coverage---we ask
that coverage holds on average over the new feature vector $X_{n+1}$, rather than requiring $\PP{\pi_P(x)\in\Ch_n(x)}\geq 1-\alpha$ to hold
uniformly over all $x\in\R^d$. Nonetheless, the main results of our work establish
that even this weak notion of distribution-free coverage is fundamentally incompatible with the goal of precise inference;
with some caveats, the main message of our results is that
the property~\eqref{eqn:distr_free_confidence_binary} can only be attained by algorithms $\Ch_n$ that return confidence intervals
whose length does not vanish with the sample size $n$.
To make this more precise, our main results are the following:
\begin{itemize}
\item {\bf Distribution-free confidence leads to distribution-free prediction.}
In Theorem~\ref{thm:prediction}, we prove that
any algorithm $\Ch_n$ satisfying~\eqref{eqn:distr_free_confidence_binary}
will inevitably also yield a valid prediction interval for $Y_{n+1}$, for any nonatomic distribution $P$,
i.e., $P$ has no point masses. This result is closely related to \citet[Proposition 5.1]{vovk2005algorithmic},
where it is shown that $\Ch_n$ must include the endpoints $0$ and/or $1$ with large probability. 
Intuitively, this implies that, in a noisy setting where $\pi_P(X)$ is not typically close to 0 or 1,
any distribution-free confidence interval $\Ch_n(X_{n+1})$ is likely to be quite wide since it needs to
reach one or both endpoints. 
\item {\bf A lower bound on the length of a distribution-free confidence interval.}
In Theorem~\ref{thm:lowerbd}, we formalize the above intuition, establishing a lower bound on the expected length
of $\Ch_n(X_{n+1})$ with an explicit function of the distribution of $\pi_P(X)$
(again, for any nonatomic $P$). Importantly, this lower bound is independent of the sample size $n$. In other words,
for any fixed nonatomic distribution $P$, the length of our distribution-free confidence intervals cannot go to zero
even as $n\rightarrow\infty$. This means that distribution-free confidence intervals cannot be adaptive---by requiring
coverage to hold for {\em all} distributions $P$, we no longer have the possibility of providing precise confidence
intervals for {\em any} distribution $P$, regardless of whether $\pi_P$ satisfies ``nice'' conditions such as smoothness.
\item {\bf A matching upper bound.}
In Theorem~\ref{thm:upperbd} we propose a concrete construction
for $\Ch_n(X_{n+1})$ that satisfies the distribution-free coverage property~\eqref{eqn:distr_free_confidence_binary}.
In particular, Corollary~\ref{cor:upperbd} proves that the length of our proposed algorithm
is asymptotically equal to the lower bound established in Theorem~\ref{thm:lowerbd},
 for any distribution $P$ where it is possible to estimate $\pi_P(X)$ consistently as $n\rightarrow\infty$.
\end{itemize}

\subsubsection{Fixed vs.~random intervals}
In some cases, we may want to allow additional randomness in our construction (formally,
we would define $\Ch_n$  as mapping a new feature vector $x$ to a distribution over subsets of $\R$,
and $\Ch_n(x)$ denotes a subset drawn from this distribution). In this
case, the probability in statements such as~\eqref{eqn:distr_free_prediction},~\eqref{eqn:distr_free_prediction_binary}, and~\eqref{eqn:distr_free_confidence_binary}
should be interpreted as being taken with respect to the distribution of the data $(X_i,Y_i)\iidsim P$ and the additional randomness in the construction of $\Ch_n(X_{n+1})$.
From this point on, we will assume that probabilities and expectations are taken on average over any randomness in the construction
of the relevant prediction or confidence interval, without further comment. In particular, all results proved in this paper apply 
to both fixed and random intervals.

\subsection{Related work}
As mentioned above, the problem of valid distribution-free confidence intervals was previously studied
by \citet[Section 5.2]{vovk2005algorithmic}.
In addition, this problem is closely related to two lines of work in the recent statistical literature---nonparametric
inference (specifically, confidence intervals for nonparametric regression), and distribution-free prediction.
\paragraph{Nonparametric confidence intervals} Suppose the response variable $Y$ 
follows a model $Y=f(X) + \textnormal{noise}$, where $f(X) = \EEst{Y}{X}$ and where the noise
distribution is constrained (e.g.,
subgaussian with some bounded variance).
In this setting, we may assume that the true regression function $f$ lies in some constrained class---for example,
it may be constrained to be Lipschitz, or to have a Lipschitz gradient (corresponding to a smoothness assumption
with exponent $\beta=1$ or $\beta=2$, respectively). There is a rich literature on the problems of estimating $f(x)$,
and providing inference (e.g., confidence bands) for $f(x)$. \edit{If the smoothness level $\beta>0$} is known, then the problem is fairly straightforward---for example,
a k-nearest neighbors method with $k\sim n^{2\beta/(2\beta+d)}$ yields the optimal estimation error rate $\mathcal{O}(n^{-\beta/(2\beta+d)})$,
ignoring log factors (and, correspondingly, confidence intervals of this length)~\citep{low1997nonparametric,gyorfi2006distribution}.

However, a key question of interest is that of {\em adaptivity}---if $\beta$ is unknown but is assumed to satisfy $\beta\geq \beta_0$
is it possible to construct confidence intervals that are valid at smoothness level $\beta_0$, but if applied to data
with smoothness $\beta>\beta_0$ would still achieve the optimal rate at that $\beta$? 
This question has been studied extensively in the literature---see, e.g., \citet{low1997nonparametric,genovese2008adaptive,cai2014adaptive}
and the references therein. It turns out that the question of adaptivity is closely tied to how we choose to define
coverage---if we require \edit{coverage at a given point $x_0$, i.e., a confidence interval for $f(x_0)$ at a fixed $x_0$,
then adaptivity is impossible \citep{low1997nonparametric}, while relaxing to nearly-uniform coverage (coverage of $f(x)$ for ``most'' points $x$)} allows for adaptivity
up to $\beta\leq 2\beta_0$ \citep{cai2014adaptive}; a bootstrap based approach for nearly-uniform coverage
is studied also by \citet{hall2013simple}. \edit{Adaptivity in the regime $\beta>2\beta_0$ can be obtained by excluding certain regions of the function space---for instance,
under the assumption that the function $f$ is either $\beta$-smooth, or is $\beta_0$-smooth and is sufficiently far from any $\beta$-smooth function, 
it becomes possible to detect the correct smoothness level of $f$ and construct the confidence band accordingly.
Results of this type are studied by \citet{carpentier2015testing,szabo2015frequentist,picard2000adaptive,hoffmann2011adaptive,gine2010confidence,bull2013adaptive}.
(An overview of many of these results can be found in \citet[Section 8.3]{gine2016mathematical}.)}

A different relaxation of the coverage condition is coverage  on average over
a random draw of $X$, studied by \citet{wahba1983bayesian}, which is similar to the coverage condition~\eqref{eqn:distr_free_confidence_binary}
studied in this work. \citet{genovese2008adaptive} propose a different relaxation, providing confidence intervals guaranteed
to cover a ``surrogate'' of the function $f$ (a smoothed version of the regression function). A different notion of coverage is the ``confidence ball'',
guaranteeing a bound $\ell_2$ error $\int_x (\widehat{f}(x) - f(x))^2\;\mathsf{d}x$ rather than providing pointwise confidence intervals for $f(x)$ at each $x$
(see, e.g., \edit{\citet{li1989honest,cai2006adaptive}}).

\paragraph{Distribution-free prediction} The field of distribution-free
prediction aims to provide prediction intervals that are uniformly valid over all distributions, without assuming
some minimum level of smoothness as in the nonparametric inference literature. 
As mentioned above, the conformal prediction framework \citep{vovk2005algorithmic} provides 
methodology for this aim.
Alternative methods offering distribution-free predictive guarantees include holdout set methods (also known 
as ``split'' or ``inductive'' conformal prediction, see, e.g., \citet{papadopoulos2002inductive,vovk2005algorithmic,papadopoulos2008inductive,lei2018distribution}),
and the jackknife+ \citep{barber2019predictive}, a variant of the jackknife (i.e., leave-one-out cross-validation).
\citet{lei2014distribution}  establish that distribution-free prediction is possible while (approximately) achieving 
the minimum possible length prediction
intervals for any ``nice'' (e.g., smooth) distribution $P$. 

The work of \citet{vovk2012conditional,lei2014distribution,barber2019limits} study whether a stronger form of predictive
coverage can be attained---namely, distribution-free conditional coverage, aiming for a guarantee that holds pointwise
at (almost) every $x$, i.e.,
$\PPst{Y_{n+1}\in\Ch_n(X_{n+1})}{X_{n+1}=x}\geq 1-\alpha$.
Distribution-free pointwise coverage is shown to be impossible for any finite-length
interval \citep{vovk2012conditional,lei2014distribution}. \citet{barber2019limits} study a weaker notion of conditional coverage,
aiming to ensure  $\PPst{Y_{n+1}\in\Ch_n(X_{n+1})}{X_{n+1}\in\Xcal}\geq 1-\alpha$ for all sufficiently large subsets $\Xcal\in\R^d$,
and prove lower bounds on the length of any resulting distribution-free interval. Some of the techniques in these lower
bounds are related to the proof techniques we use in the present work.

In the setting where the response $Y$ is binary or takes finitely many values, as discussed earlier,
\citet{vovk2005algorithmic,lei2014classification,sadinle2019least} apply the conformal prediction
framework to the problem of distribution-free classification. If the goal is to estimate label probabilities (rather than output a predictive set),
an alternative notion of validity in the binary setting is {\em calibration}, where for an estimate $\widehat{\pi}(X)$
of the label probability $\pi_P(X)$, 
we require
$\PPst{Y=1}{\widehat{\pi}(X)} = \widehat{\pi }(X)$. This framework is studied in the distribution-free
setting by \citet{vovk2014venn} via the methodology of Venn predictors, and more recently by \citet{gupta2020distribution}.

\section{Main results: lower bounds}
In this section, we will prove that a distribution-free confidence interval for binary regression cannot 
provide precise inference about the parameter $\pi_P(X)$. To do this, we first compare to the problem of predictive inference, and then
turn to proving lower bounds on the length of any distribution-free confidence interval.

\subsection{Confidence vs.~prediction}
Our first main result proves that, in the binary regression setting, any algorithm providing distribution-free coverage
of $\pi_P(X)$ must necessarily also cover the binary label $Y$, for every distribution $P$ that is nonatomic (i.e., zero probability
at any single point).
\begin{theorem}\label{thm:prediction}
Let $\Ch_n$ be any algorithm that provides a $(1-\alpha)$-distribution-free confidence interval for
binary regression, as in~\eqref{eqn:distr_free_confidence_binary}.
Then $\Ch_n$ also satisfies $(1-\alpha)$ predictive coverage uniformly over all nonatomic distributions $P$. That is, $\Ch_n$
satisfies
\[\Pp{(X_i,Y_i)\iidsim P}{Y_{n+1}  \in \Ch_n(X_{n+1})} \geq 1-\alpha\textnormal{\ for
 all nonatomic distributions $P$ on $\R^d\!\times\!\{0,1\}$}.\]
\end{theorem}
\noindent For example, consider a distribution $P$ with a constant label probability, $\pi_P(x)\equiv 0.5$.
Given a large sample size $n$, we might hope that our algorithm would detect the simple nature of this distribution, and could
output a narrow interval, $\Ch_n(X_{n+1}) = 0.5\pm \mathrm{o}(1)$. However, Theorem~\ref{thm:prediction} tells us that any
distribution-free confidence interval $\Ch_n$ must necessarily include both endpoints $0$ and $1$ with substantial probability.
In particular, this example suggests that, unless $\pi_P(X)$ is usually close to 0 or 1, any distribution-free confidence
interval $\Ch_n$ is unlikely to be precise (i.e., $\Ch_n(X_{n+1})$ is unlikely to be a short interval). In Theorem~\ref{thm:lowerbd}
below, we will formalize this intuition
by finding a lower bound on the expected length of $\Ch_n(X_{n+1})$.

We note that Theorem~\ref{thm:prediction} is closely related to a result of \citet[Proposition 5.1]{vovk2005algorithmic}
(see also \citet[Theorem 7, Corollary 18]{nouretdinov2001pattern} for an earlier related result).
Their work establishes that if $\Ch_n$ provides a $(1-\alpha)$-distribution-free confidence interval for
binary regression~\eqref{eqn:distr_free_confidence_binary}, then there exists some other algorithm $\widetilde C_n$, 
also satisfying the property~\eqref{eqn:distr_free_confidence_binary},
such that 
\[  \PP{\Ch_n(X_{n+1}) \subseteq(0,1)}\leq \PP{\widetilde C_n(X_{n+1}) = \emptyset} .\]
Clearly this last quantity cannot be larger than $\alpha$ (since this would immediately contradict
the coverage property~\eqref{eqn:distr_free_confidence_binary}), and so
this result, like Theorem~\ref{thm:prediction} above, indicates that $\Ch_n(X_{n+1})$ must often
 include endpoints $0$ and/or $1$.
 While \citet{vovk2005algorithmic}'s result appears different from the conclusion
of Theorem~\ref{thm:prediction} above, the construction in their proof in fact suffices to prove Theorem~\ref{thm:prediction} as well.

\subsubsection{A key lemma}
Rather than proving Theorem~\ref{thm:prediction} directly, we will instead generalize to a more powerful result:
\begin{lemma}\label{keylemma}
Let $\Ch_n$ be any algorithm that provides a $(1-\alpha)$-distribution-free confidence interval for 
binary regression~\eqref{eqn:distr_free_confidence_binary}.
Let $P$ be any nonatomic distribution on $(X,Y)\in\R^d\times\{0,1\}$.
Then for data points $(X_i,Y_i)$ drawn i.i.d.~from $P$ and any random variable $Z_{n+1}\in[0,1]$ satisfying
 \[Z_{n+1}\independent \Ch_n(X_{n+1}) \mid X_{n+1}\textnormal{\quad and \quad}
\EEst{Z_{n+1}}{X_{n+1}} = \pi_P(X_{n+1})\textnormal{ almost surely},\]
it holds that
\[\PP{Z_{n+1}\in\Ch_n(X_{n+1})}\geq 1-\alpha.\]\end{lemma}
\noindent (Proofs for this lemma and all subsequent theoretical results 
are deferred to the Appendix.)

With this lemma in place, Theorem~\ref{thm:prediction} follows immediately---namely,
defining $Z_{n+1}=Y_{n+1}$, we have proved the theorem. 
The lemma is substantially more general, however, and we will need
its full generality in order to prove our lower bounds on the length of $\Ch_n(X_{n+1})$ below.

\subsection{A lower bound on length}
Next, we will establish bounds on the length of a distribution-free confidence interval for binary regression.
We begin with a few definitions. First, for $t\in[0,\frac{1}{2}]$ and $a\in[0,1]$, we define
\[\ell(t,a) = \begin{cases}
2(1-a)t,& a\geq\frac{1}{2},\\
\frac{t}{2a}, &a\geq t \textnormal{ and }0<a<\frac{1}{2},\\
1-\frac{a}{2t}, &a <t,\\
0,&a = t=0,\end{cases}\]
and for $t\in(\frac{1}{2},1]$ let $\ell(t,a) = \ell(1-t,a)$.
The function $\ell(t,a)$ is illustrated in Figure~\ref{fig:ell_t_a_plots}.
To understand the role of this function in our work, we begin with the following lemma:
\begin{lemma}\label{lem:ell_is_inf}For any $t,a\in[0,1]$, define
\[\Fcal_{t,a} = \left\{\textnormal{\begin{tabular}{c}Measurable functions $f:[0,1]\rightarrow[0,1]$ satisfying\\ $\EE{f(Z)}\geq 1-a$ for any 
 random variable $Z\in[0,1]$ with $\EE{Z}=t$\end{tabular}}\right\}.\]
 Then it holds that
\[\ell(t,a) = \inf_{f\in\Fcal_{t,a}}\left\{\int_{s=0}^1 f(s)\;\mathsf{d}s \right\}.\]
\end{lemma}
\noindent Next, for any distribution $Q$ on $[0,1]$ and any $\alpha\in[0,1]$, define 
\[L_\alpha(Q) = \inf_{\substack{\textnormal{Measurable fns.}\\a:[0,1]\rightarrow[0,1]}}\Big\{\Ep{T\sim Q}{\ell(T,a(T))} : \Ep{T\sim Q}{a(T)} \leq \alpha\Big\}.\]
In the following theorem, we will see that this function allows us to explicitly characterize
lower and upper bounds for the length of any distribution-free confidence interval.

\begin{figure}[t]
\centering
\includegraphics[width=\textwidth]{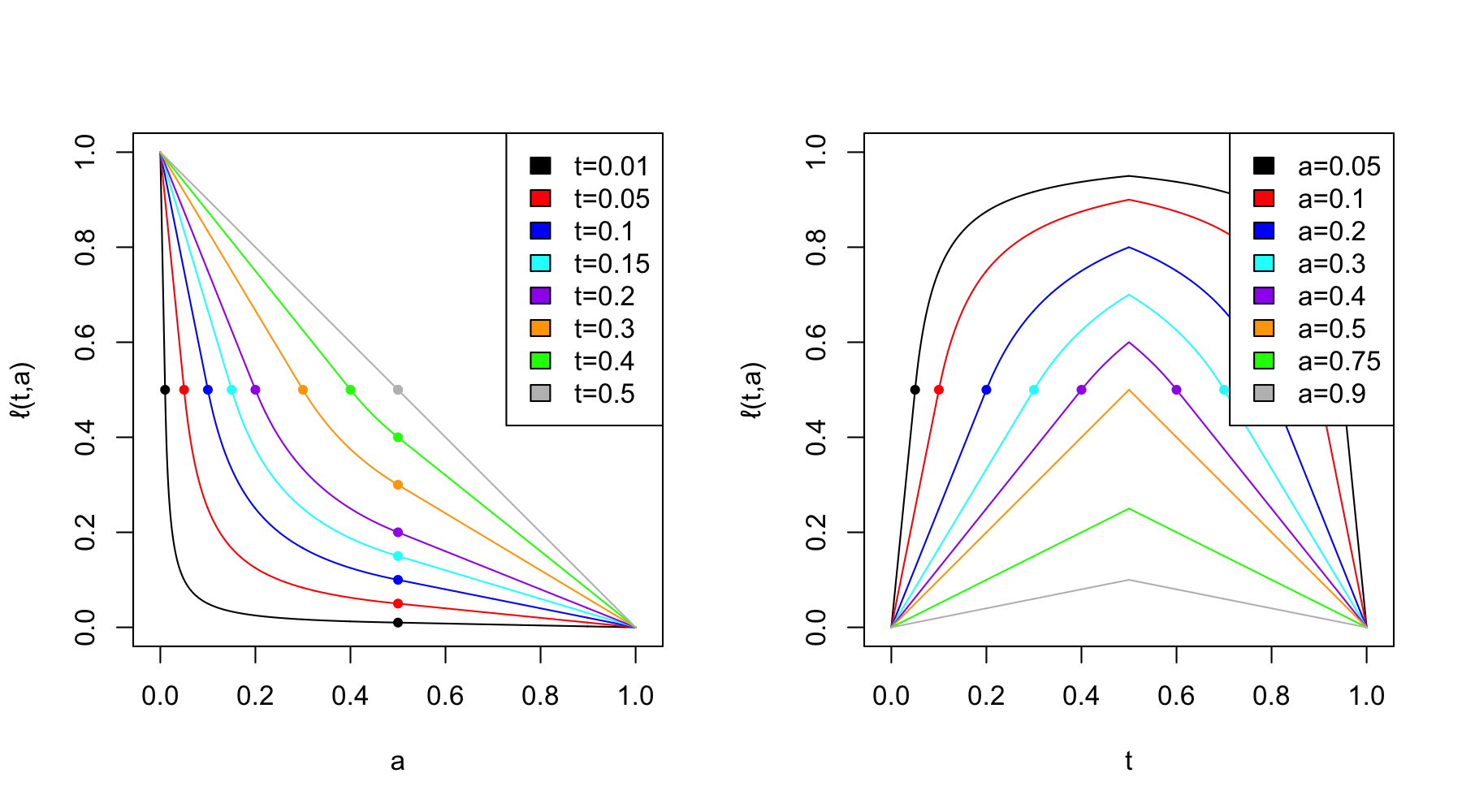}
\caption{Illustration of the function $\ell(t,a)$. The dots indicate the points 
where the function switches between the three cases in its definition (i.e., $a=\frac{1}{2}$ and $a=\min\{t,1-t\}$).
}
\label{fig:ell_t_a_plots}
\end{figure}

\begin{theorem}\label{thm:lowerbd}
Let $\Ch_n$ be any algorithm that provides a $(1-\alpha)$-distribution-free confidence interval for binary
regression~\eqref{eqn:distr_free_confidence_binary}.
Then for any nonatomic distribution $P$ on $\R^d\times\{0,1\}$, it holds that
\[\Ep{(X_i,Y_i)\iidsim P}{\leb(\Ch_n(X_{n+1}))} \geq L_\alpha(\Pi_P),\]
where $\Pi_P$ is the distribution of the random variable $\pi_P(X)\in[0,1]$.
\end{theorem}
\noindent Here $\leb()$ denotes the Lebegue measure on $\R$.

\subsubsection{Proof sketch for Theorem~\ref{thm:lowerbd}}
First, for intuition, we can consider a simple case where $\Pi_P = \delta_t$, a point mass at some $t\in[0,1]$---that is, the label probability is constant,
with $\pi_P(x)=t$ for all $x$. Define a function $f:[0,1]\rightarrow[0,1]$ as $f(s) = \Pp{(X_i,Y_i)\iidsim P}{s\in\Ch_n(X_{n+1})}$.
By Lemma~\ref{keylemma}, 
for any random variable $Z\in[0,1]$ with $\EE{Z}=t$ drawn independently of the data, it holds that
\[1-\alpha \leq \PP{Z\in\Ch_n(X_{n+1})} = \EE{\PPst{Z\in\Ch_n(X_{n+1})}{Z}} = \EE{f(Z)}.\]
Applying Lemma~\ref{lem:ell_is_inf}, we therefore  have 
\begin{multline*} \ell(t,\alpha) \leq \int_{s=0}^1 f(s)\;\mathsf{d}s= \int_{s=0}^1 \Pp{(X_i,Y_i)\iidsim P}{s\in \Ch_n(X_{n+1})}\;\mathsf{d}s\\
= \Ep{(X_i,Y_i)\iidsim P}{\int_{s=0}^1 \One{s\in \Ch_n(X_{n+1})}\;\mathsf{d}s} = \EE{\leb(\Ch_n(X_{n+1}))} ,\end{multline*}
by Fubini's theorem.
We can also verify that $L(\Pi_P) =L(\delta_t)= \ell(t,\alpha)$, completing the proof for this simple case.

Now, in general, $\pi_P(x)$ will not be a constant. Consider any distribution-free confidence interval $\Ch_n$,
and let $a_P(t)$ be the noncoverage rate over test points $X_{n+1}$ conditional on $\pi_P(X_{n+1})=t$, for a particular distribution $P$:
\[a_P(t) = \Ppst{(X_i,Y_i)\iidsim P}{\pi_P(X_{n+1})\not\in\Ch_n(X_{n+1})}{\pi_P(X_{n+1})=t}.\]
We must therefore have $\Ep{T\sim \Pi_P}{a_P(T)}\leq \alpha$, in order to achieve at least $1-\alpha$ coverage.
By comparing to the constant-probability case, informally we can see that $\ell(t,a_P(t))$ must be a lower bound
on $\EEst{\leb(\Ch_n(X_{n+1}))}{\pi_P(X_{n+1})=t}$.
Therefore, marginalizing over $\pi_P(X_{n+1})$, we see that $\EE{\leb(\Ch_n(X_{n+1}))}$ is lower bounded by  $\Ep{T\sim \Pi_P}{\ell(T,a_P(T))}$,
which is $\geq L_\alpha(\Pi_P)$ by definition.
The full details of this proof are deferred to the Appendix.

\subsubsection{Interpreting the lower bound}
As mentioned in the proof sketch above, we have seen that
\[\textnormal{If $\pi_P(X)=t$ almost surely, then $L_\alpha(\Pi_P) = \ell(t,\alpha)$,}\]
giving us an exact expression for the lower bound on length in the case where the label probability is constant.

More generally, we can verify that, for any $t\in[0,\frac{1}{2}]$,
\[\textnormal{If $\pi_P(X)\in[t,1-t]$ almost surely, then $L_\alpha(\Pi_P) \geq \ell(t,\alpha)$}.\]
(This bound holds because, for any $s\in[t,1-t]$, we have $\ell(s,a)\geq \ell(t,a)$ for all $a$; examining the definition of $L_\alpha$ leads immediately to this lower bound.) 
This inequality means that, if $\pi_P(X)$ is bounded away from 0 and 1, then there is a fundamental lower bound on 
the length of any distribution-free confidence interval regardless of the sample size $n$, since
the lower bound $\ell(t,\alpha)>0$ does not depend on $n$.
In other words, an infinite sample size does not lead to infinite precision.

\subsubsection{Comparison to existing results in predictive inference}
\citet[Section 3.4]{vovk2005algorithmic} study distribution-free prediction in a setting where
the label $Y$ takes values in a finite set $\mathcal{Y}$, with binary labels $Y\in\{0,1\}$ as a special case.
Their results (specifically, see \citet[Propositions 3.3--3.5]{vovk2005algorithmic}) characterize the minimum possible expected cardinality
of any distribution-free predictive set in terms of the ``predictability'' of $Y$ given $X$---in the special case of a binary label, this
translates to studying the distribution $\Pi_P$ of $\pi_P(X)$. However, the two problems (predictive subsets of $\{0,1\}$ versus
confidence intervals that are subsets of $[0,1]$) are very different in nature, and their results are not
directly related to the map $L_\alpha(\Pi_P)$ derived in our work.

\section{Main results: upper bounds}
We will next investigate whether the lower bound on confidence interval length, proved in Theorem~\ref{thm:lowerbd},
 can in fact be achieved by a distribution-free method. In order to be able to
construct confidence intervals based on a finite sample, we will work in a setting where we approximate $P$ via a partition.

We begin by defining some notation. 
First suppose we are given a predefined partition $\R^d=\Xcal_1\cup\dots\cup\Xcal_M$.
(Later on, for a distribution-free algorithm, we will allow the partition to be chosen as a function of the data.)
For each $x\in\R^d$, we 
define $m(x)$ to be the index of the region containing $x$, i.e., $\Xcal_{m(x)}\ni x$.
\edit{For each $m$, we define
\[p_{P,m} = \Pp{P_X}{X\in\Xcal_m},\]
the probability of $X$ lying in the $m$th region, and
\[\pi_{P,m} = \Epst{P_X}{\pi_P(X)}{X\in\Xcal_m} = \Ppst{P}{Y=1}{X\in\Xcal_m},\]
 the average label probability within the $m$th region.}
 
We will consider confidence intervals that pool data within each region---in particular, we will construct a
confidence interval $\Ch_n(X_{n+1})$ that depends on $X_{n+1}$ only through $m(X_{n+1})$.
As a result, it is clear that
we will only be able to produce a precise confidence interval if $\pi_P(x)$ is approximately constant over $x\in\Xcal_m$,
for each $m$. To capture this notion, we define the ``blur'' of the partition $\Xcal_{1:M} = \{\Xcal_m\}_{m=1,\dots,M}$ as 
\[\Delta_P(\Xcal_{1:M}) = \Ep{P_X}{|\pi_P(X) - \pi_{P,m(X)}|}.\]
In other words, $\Delta_P(\Xcal_{1:M})$ will be low if the partition is highly informative, separating $\R^d$ into regions
where $\pi_P(x)$ is nearly constant. For example, if we have access to a good estimate of the function $\pi_P(x)$, we can 
partition $\R^d$ by clustering together points with similar estimated probabilities.
Of course, the blur $\Delta_P(\Xcal_{1:M})$ cannot in general be small unless we choose $M$ to be large. We will discuss this 
tradeoff later on.

\subsection{An oracle algorithm}
In order to motivate our distribution-free construction, we begin with a simpler problem.
Suppose that we are given a  fixed partition $\R^d=\Xcal_1\cup\dots\cup\Xcal_M$, and are given oracle knowledge
of the probabilities $p_{P,m}$ and $\pi_{P,m}$ defined above. What is the best possible interval length that can be obtained
using this oracle knowledge?
 
As for the lower bound, let us begin by examining the function $\ell(t,a)$. For any $t,a\in[0,1]$, define
a function $f_{t,a}:[0,1]\rightarrow[0,1]$ as follows. If $t\in[0,\frac{1}{2}]$, define
 \begin{equation}\label{eqn:f_t_a}f_{t,a}(s) = \begin{cases}
(1-a)\cdot \one{s=0},&\textnormal{ if $t=0$},\\
(1-a),&\textnormal{ if $t=\frac{1}{2}$},\\
2(1-a)\cdot\max\{1-\frac{s}{2t},0\},&\textnormal{ if $0<t<\frac{1}{2}$ and $a\geq\frac{1}{2}$,}\\
\max\{1-\frac{as}{t},0\},&\textnormal{ if $0<t<\frac{1}{2}$ and $a<\frac{1}{2}$,}
\end{cases}\end{equation} 
and for $t\in(\frac{1}{2},1]$ define $f_{t,a}(s) = f_{1-t,a}(1-s)$.
 In the proof of Lemma~\ref{lem:ell_is_inf}, we will establish that $f_{t,a}$ satisfies $\int_{s=0}^1f_{t,a}(s)\;\mathsf{d}s=\ell(t,a)$,
 and that $\EE{f_{t,a}(Z)}\geq 1-a$ for any $Z$ with $\EE{Z}=t$. In other words, the function $f_{t,a}$ attains the infimum 
 in the statement of that lemma.

Next we will leverage this function to construct a confidence interval using the given partition.
Fix any $\bt=(t_1,\dots,t_M)\in[0,1]^M$ and any $\ba=(a_1,\dots,a_M)\in[0,1]^M$.
Given a test point $x\in\R^d$, we first draw an independent random variable $U\sim\textnormal{Unif}[0,1]$, and then define
\begin{equation}\label{eqn:construction_generic}C_{\bt,\ba}(x) = \left\{s\in[0,1] : f_{t_{m(x)},a_{m(x)}}(s) \geq U\right\}.\end{equation}
The following lemma examines the length and coverage properties of this construction:
\begin{lemma}\label{lem:generic_a_t}
Consider a fixed partition $\R^d=\Xcal_1\cup\dots\cup\Xcal_M$, and a fixed $\bt=(t_1,\dots,t_M)\in[0,1]^M$ and $\ba=(a_1,\dots,a_M)\in[0,1]^M$.
Then for any distribution $P$ on $(X,Y)\in\R^d\times\{0,1\}$,
\[\Ep{P_X}{\leb(C_{\bt,\ba}(X))} =\sum_{m=1}^M p_{P,m}\ell(t_m,a_m).\]
If additionally it holds that
\begin{equation}\label{eqn:pi_vs_t}
\textnormal{For all $m$, either $0\leq \pi_{P,m}\leq t_m\leq \frac{1}{2}$ or $\frac{1}{2}\leq t_m\leq \pi_{P,m}\leq 1$},\end{equation}
then\[\Pp{P_X}{\pi_P(X)  \in C_{\bt,\ba}(X)} \geq 1-\sum_{m=1}^M p_{P,m} a_m.\]
\end{lemma}
\noindent (Note that $C_{\bt,\ba}(X)$ is a randomized interval,
and so the probability and expectation are taken with respect to the data point $X\sim P_X$
and the independent random variable $U\sim\textnormal{Unif}[0,1]$ used in the construction of $C_{\bt,\ba}(X)$.)

Next, we define the oracle confidence interval.
Suppose we are  given oracle knowledge of the $p_{P,m}$'s and $\pi_{P,m}$'s. Define
\begin{equation}\label{eqn:a_star}\ba^*_P=(a^*_{P,1},\dots,a^*_{P,M}) \in\argmin_{a_1,\dots,a_M\in [0,1]} \left\{\sum_{m=1}^M p_{P,m} \ell(\pi_{P,m},a_m) : \sum_{m=1}^M p_{P,m} a_m \leq \alpha\right\}.\end{equation}
(This is a convex optimization problem, since $\ell(t,a)$ is convex in $a$.)
Given a test point $x\in\R^d$, we define the oracle interval as
\begin{equation}\label{eqn:construction_oracle}
C^*_P(x) = C_{\bpi_P,\ba^*_P}(x),\end{equation}
where we write $\bpi_P = (\pi_{P,1},\dots,\pi_{P,M})$.
To  understand this oracle interval, by the results
of Lemma~\ref{lem:generic_a_t}, we can observe that the constraint $\sum_{m=1}^M p_{P,m}a_m\leq \alpha$
ensures $1-\alpha$ coverage for the distribution $P$, while minimizing $\sum_{m=1}^M p_{P,m} \ell(\pi_{P,m},a_m)$
ensures the lowest possible length under this coverage constraint.

Our next result shows that if the blur $\Delta_P(\Xcal_{1:M})$ of the partition is low, 
then the expected length of $C^*_P$ is close to the distribution-free lower bound $L_\alpha(\Pi_P)$.
\begin{lemma}\label{lem:upperbd_oracle} The oracle interval  $C^*_P$ constructed in~\eqref{eqn:construction_oracle}
satisfies
\[\Pp{P_X}{\pi_P(X)  \in C^*_P} \geq 1-\alpha\]
and
\[
\Ep{P_X}{\leb(C^*_P(X))} \leq L_\alpha(\Pi_P) +\sqrt{\frac{2\Delta_P(\Xcal_{1:M})}{\alpha}}.\]
\end{lemma}
\noindent Of course, $C^*_P$ is {\em not} a distribution-free confidence interval---its coverage is guaranteed
for a specific distribution $P$ (not uniformly over all $P$), with the assumption
that we have information about the distribution---namely, the $p_{P,m}$'s and $\pi_{P,m}$'s. We next
extend this construction into the distribution-free setting by using the training sample to estimate these quantities.

\subsection{A distribution-free algorithm}
We are now ready to present our distribution-free algorithm.
For now,  assume again that we are given a fixed partition $\R^d=\Xcal_1\cup\dots\cup\Xcal_M$ 
(we will allow for a data-dependent partition later on).
After observing a sample $(X_1,Y_1),\dots,(X_n,Y_n)\in\R^d\times\{0,1\}$,
we first estimate the probability $p_{P,m}$ in each region,
\[\widehat p_m =\frac{\sum_{i=1}^n\One{X_i\in\Xcal_m}}{n},\]
and estimate the corresponding label probability $\pi_{P,m}$,
\[\widehat\pi_m =\frac{\sum_{i=1}^n \One{X_i \in \Xcal_m, Y_i = 1}}{n\widehat p_m},\]
or set $\widehat\pi_m = \frac{1}{2}$ if $\widehat p_m = 0$.
In order to ensure distribution-free coverage we will need to work with slightly more conservative estimates.
 Let
\begin{equation}\label{eqn:p_tilde}\tilde p_m =  \widehat p_m + \sqrt{\widehat p_m\cdot \frac{3\log(4Mn/\alpha)}{n}}+\frac{3\log(4Mn/\alpha)}{n},\end{equation}
which is chosen to ensure that $p_m\leq \tilde p_m$ with high probability, and let
\begin{equation}\label{eqn:pi_tilde}\tilde\pi_m = \begin{cases}
\min\left\{\frac{1}{2}, 
\widehat\pi_m + \sqrt{\widehat\pi_m\cdot \frac{2\log(4Mn/\alpha)}{n\widehat p_m}} + \frac{2\log(4Mn/\alpha)}{n\widehat p_m}\right\}, & \textnormal{ if $\widehat\pi_m\leq\frac{1}{2}$},\\
\max\left\{\frac{1}{2}, 
\widehat\pi_m - \sqrt{(1-\widehat\pi_m)\cdot \frac{2\log(4Mn/\alpha)}{n\widehat p_m}} - \frac{2\log(4Mn/\alpha)}{n\widehat p_m}\right\}, & \textnormal{ if $\widehat\pi_m>\frac{1}{2}$}.
\end{cases}\end{equation}
This definition of $\tilde\pi_m$ is designed to pull the estimate $\widehat\pi_m$ closer to $\frac{1}{2}$ (since $\frac{1}{2}$ is the most challenging
label probability for coverage, this is therefore a more conservative estimate of $\pi_{P,m}$).

From this point on, we use the same construction~\eqref{eqn:construction_generic} as for the ``oracle'' interval $C^*_P$ above,
but with the conservative empirical estimates $\tilde p_m$ and $\tilde\pi_m$ in place of the unknown true quantities $p_{P,m}$ and $\pi_{P,m}$.
Define
 \begin{equation}\label{eqn:a_tilde}\tilde\ba = (\tilde a_1,\dots,\tilde a_M)  \in\argmin_{a_1,\dots,a_M\in [0,1]} \left\{\sum_{m=1}^M \tilde p_m \ell(\tilde\pi_m,a_m) : \sum_{m=1}^M \tilde p_m a_m \leq \alpha\right\},\end{equation}
and define
\begin{equation}\label{eqn:construction}\Ch_n(X_{n+1}) = C_{\tilde\bpi,\tilde\ba}(X_{n+1}),\end{equation}
where we write $\tilde\bpi = (\tilde\pi_1,\dots,\tilde\pi_M)$.

We now prove that this construction offers a distribution-free confidence interval, 
and establish an upper bound on its expected length. 
\begin{theorem}\label{thm:upperbd} Let $n\geq 2$.
The confidence interval $\Ch_n$ constructed in~\eqref{eqn:construction}
provides a $(1-\alpha)$-distribution-free confidence interval for 
binary
regression~\eqref{eqn:distr_free_confidence_binary}.
Furthermore, for all distributions $P$ on $\R^d\times\{0,1\}$, the confidence interval $\Ch_n$ satisfies
\[\Ep{(X_i,Y_i)\iidsim P}{\leb(\Ch_n(X_{n+1}))} \leq L_\alpha(\Pi_P) + \sqrt{\frac{2\Delta_P(\Xcal_{1:M})}{\alpha}}+ c\sqrt{\frac{M\log n}{\alpha n}},\]
where $c$ is a universal constant.
\end{theorem}
\noindent 
Comparing to the upper bound calculated in Lemma~\ref{lem:upperbd_oracle} for the oracle confidence interval $C^*_P$, 
the only additional term is $c\sqrt{\frac{M\log n}{\alpha n}}$, which is vanishing with $n\rightarrow \infty$ as long as the partition size 
$M$ is sufficiently small relative to the sample size $n$.

\subsubsection{Sample splitting for a data-dependent partition}
Thus far we have assumed that the partition $\R^d=\Xcal_1\cup\dots\cup\Xcal_M$ is fixed (and has low
blur $\Delta_P(\Xcal_{1:M})$, in order for the upper bound to be meaningful).
Of course, in practice, the partition itself would need to be constructed using the data.
To do so, we follow a sample splitting strategy. We will use the first half of the training data
to estimate the function $\pi_P(x)$ and define the partition accordingly, and the second half of the training data to construct $\Ch_n$
based on this partition.
Our construction can be paired with any regression algorithm $\Rcal$ that maps a training data set of size $n$ to an estimate $\widehat{\pi}^{\Rcal}_n(x)$
of the function $\pi_P(x)$---for example, we might take $\Rcal$ to be logistic regression or k-nearest neighbors.
We define the expected error of this regression algorithm as:
\[\Delta_{n,P}(\Rcal) =  \EE{\big|\widehat\pi^{\Rcal}_n(X_{n+1}) - \pi_P(X_{n+1})\big|},\]
where the expected value is taken over $(X_1,Y_1),\dots,(X_n,Y_n)\iidsim P$ and an independent test point $X_{n+1}\sim P_X$.
We mention a few special cases:
\begin{itemize}
\item In a well-specified parametric model (e.g., $Y|X$ follows a logistic model), we typically have $\Delta_{n,P}(\Rcal) \sim \sqrt{\frac{d}{n}}$
in a low-dimensional ($d<n$) setting, or $\Delta_{n,P}(\Rcal) \sim \sqrt{\frac{k\log d}{n}}$ in a high-dimensional sparse setting, e.g., running $\ell_1$-penalized
logistic regression when the true model is $k$-sparse~\citep{negahban2012unified}.
\item In a nonparametric setting, if $x\mapsto\pi_P(x)$ is assumed to be Lipschitz continuous, a k-nearest neighbors (k-NN)
method yields $\Delta_{n,P}(\Rcal) \sim n^{-1/(2+d)}$, or $\Delta_{n,P}(\Rcal) \sim n^{-2/(4+d)}$  if we make the stronger assumption that $x\mapsto\pi_P(x)$ is smooth (see, e.g., \citet{gyorfi2006distribution}).
 In the special case where the data is supported on a $d_0$-dimensional manifold in $\R^d$ for some $d_0<d$,
then the bound may hold with $d_0$ in place of $d$ for a faster convergence rate (see, e.g., \citet{jiang2019non} for finite sample results). 
\end{itemize}

Now we split the sample to construct our distribution-free confidence
interval. First, define $\widehat\pi^{\Rcal}_{\lfloor \frac{n}{2}\rfloor}$, fitted on data points $i=1,\dots,\lfloor \frac{n}{2}\rfloor$.
 Fix
$M = \lceil \sqrt{n /\log n}\rceil$,
and define
\[\widehat{\Xcal}^{\Rcal}_1= \left\{x\in\R^d : 0 \leq \widehat\pi^{\Rcal}_{\lfloor \frac{n}{2}\rfloor}(x) < \frac{1}{M}\right\} \ , \
 \dots \ , \ \widehat{\Xcal}^{\Rcal}_M= \left\{x\in\R^d : \frac{M-1}{M}\leq \widehat\pi^{\Rcal}_{\lfloor \frac{n}{2}\rfloor}(x) \leq1\right\}.\]
With this partition in place, we then define $\Ch_n$ exactly as before, except that we restrict our sample to the
remaining data points $i=\lfloor\frac{n}{2}\rfloor+1,\dots,n$. Specifically,
let
\[\widehat p_m =\frac{\sum_{i=\lfloor \frac{n}{2}\rfloor+1}^n\One{X_i\in\widehat{\Xcal}^{\Rcal}_m}}{\lceil\frac{n}{2}\rceil}\textnormal{\quad and \quad}
\widehat\pi_m =\frac{\sum_{i=\lfloor \frac{n}{2}\rfloor+1}^n \One{X_i \in \widehat{\Xcal}^{\Rcal}_m, Y_i = 1}}{\lceil\frac{n}{2}\rceil\widehat p_m},\]
or set $\widehat\pi_m = \frac{1}{2}$ if $\widehat p_m = 0$. Define the $\tilde p_m$'s and $\tilde\pi_m$'s exactly as in~\eqref{eqn:p_tilde}
and~\eqref{eqn:pi_tilde} except with $\lceil\frac{n}{2}\rceil$ in place of $n$. We then define $\tilde\ba$ as in~\eqref{eqn:a_tilde}
before and set
\begin{equation}\label{eqn:construction_split}\Ch_n^{\Rcal}(X_{n+1}) = C_{\tilde\bpi,\tilde\ba}(X_{n+1}).\end{equation}
\begin{corollary}\label{cor:upperbd}
Let $n\geq 3$. 
The interval $\Ch_n^{\Rcal}$ constructed via data splitting with regression algorithm $\Rcal$~\eqref{eqn:construction_split}
provides a $(1-\alpha)$-distribution-free confidence interval for 
binary
regression~\eqref{eqn:distr_free_confidence_binary}.
Furthermore, for all distributions $P$ on $\R^d\times\{0,1\}$, $\Ch_n^{\Rcal}$ 
satisfies
\[\Ep{(X_i,Y_i)\iidsim P}{\leb(\Ch_n^{\Rcal}(X_{n+1}))} \leq L_\alpha(\Pi_P) + 2\sqrt{\frac{\Delta_{\lfloor \frac{n}{2}\rfloor,P}(\Rcal)}{\alpha}} +\frac{c'}{\sqrt{\alpha}}\sqrt[4]{\frac{\log n}{n}},\]
where $c'$ is a universal constant.
\end{corollary}
\noindent This corollary will follow directly from Theorem~\ref{thm:upperbd} by observing that, due to the construction of the 
random partition $\widehat{\Xcal}^{\Rcal}_{1:M}$,
we have $\EE{\Delta_P(\widehat{\Xcal}^{\Rcal}_{1:M})} \leq 2\Delta_{\lfloor \frac{n}{2}\rfloor,P}(\Rcal)+\frac{1}{M}$.

\section{Discussion}
The lower bounds established in this paper prove that, in the distribution-free setting,
parameter estimation is fundamentally as imprecise as prediction, and confidence intervals 
for estimating the label probabilities $\pi_P(X) = \PPst{Y=1}{X}$ have a lower bound on their length
that does not vanish even with sample size $n\rightarrow\infty$. Unlike the classical literature where these
types of results are established for pointwise coverage (i.e., coverage of $\pi_P(x)$ for all $x$),
our new results prove this fundamental lower bound holds even when we require coverage to hold
only on average over a new point $X$ drawn from the distribution, and we provide an exact calculation
of the minimum possible length.
These lower bounds imply that, if we wish to maintain the versatility of the distribution-free setting (i.e., avoiding
smoothness assumptions), we can only obtain meaningful confidence intervals by substantially relaxing
our notion of coverage. In future work, we hope to examine alternatives---for instance, coverage
of a surrogate function approximating $\pi_P(X)$, as in the work of \citet{genovese2008adaptive} for confidence bands
in nonparametric regression.

We may also ask whether the results here, proved in the setting of binary regression where $Y\in\{0,1\}$,
may be extended to a more general regression setting---that is, whether it is possible to estimate $\mu_P(X)=\EEst{Y}{X}$
under a joint distribution $P$ over $(X,Y)\in\R^d\times\R$. An initial exploration suggests that 
such an extension would not be straightforward. Specifically, if $Y$ is unbounded, then confidence intervals for $\mu_P(X)$
must necessarily have  infinite expected length. To see why, consider a contamination model, 
where, after drawing $(X,Y)\sim P$, with probability $\eps_n$ we replace $Y$ with a corrupted
value $c_n$. If we choose $\eps_n\ll n^{-1}$, we cannot distinguish between $P$ and this contaminated model $P'$
on a sample of size $n$---this means that $\Ch_n(X_{n+1})$ must
cover the mean whether the distribution is $P$ or $P'$. But, by choosing $c_n$ so that $c_n\eps_n\rightarrow\infty$, the means $\mu_P(X)$ and
$\mu_{P'}(X)$ are arbitrarily far apart, leading to arbitrarily large length of $\Ch_n(X_{n+1})$. Therefore
we cannot obtain nontrivial results for unbounded $Y$. If we instead assume that $Y$ is bounded, on the other hand,
then the lower bounds established for the binary case no longer apply (for example, if $Y=0.5$ almost surely,
a distribution-free confidence interval could potentially have vanishing length even if it is constructed only assuming $Y\in[0,1]$).
It is therefore not clear whether there are settings for the general regression problem where we may obtain meaningful
upper and lower bounds for distribution-free confidence intervals on the conditional mean $\mu_P(X)$, and we leave these questions
for future work.

\subsection*{Acknowledgements}
R.F.B.~was supported by the National Science Foundation via grant DMS--1654076, and by the Office of Naval
Research via grant N00014-20-1-2337. The author thanks
Emmanuel Cand{\`e}s, Haoyang Liu, Aaditya Ramdas, and Ryan Tibshirani for inspiration and many helpful discussions.
\bibliographystyle{plainnat}
\bibliography{bib}

\begin{thebibliography}{32}
\providecommand{\natexlab}[1]{#1}
\providecommand{\url}[1]{\texttt{#1}}
\expandafter\ifx\csname urlstyle\endcsname\relax
  \providecommand{\doi}[1]{doi: #1}\else
  \providecommand{\doi}{doi: \begingroup \urlstyle{rm}\Url}\fi

\bibitem[Aliprantis and Border(2006)]{aliprantis2006infinite}
Charalambos~D Aliprantis and Kim~C Border.
\newblock \emph{Infinite Dimensional Analysis: A Hitchhiker's Guide}.
\newblock Springer, 2006.

\bibitem[Barber et~al.(2019)Barber, Cand{\`e}s, Ramdas, and
  Tibshirani]{barber2019predictive}
Rina~Foygel Barber, Emmanuel~J Cand{\`e}s, Aaditya Ramdas, and Ryan~J
  Tibshirani.
\newblock Predictive inference with the jackknife+.
\newblock \emph{arXiv preprint arXiv:1905.02928}, 2019.

\bibitem[Barber et~al.(2020)Barber, Cand{\`e}s, Ramdas, and
  Tibshirani]{barber2019limits}
Rina~Foygel Barber, Emmanuel~J Cand{\`e}s, Aaditya Ramdas, and Ryan~J
  Tibshirani.
\newblock The limits of distribution-free conditional predictive inference.
\newblock \emph{Information and Inference: A Journal of the IMA}, 2020.
\newblock \doi{10.1093/imaiai/iaaa017}.

\bibitem[Bull and Nickl(2013)]{bull2013adaptive}
Adam~D Bull and Richard Nickl.
\newblock Adaptive confidence sets in ${L}^2$.
\newblock \emph{Probability Theory and Related Fields}, 156\penalty0
  (3-4):\penalty0 889--919, 2013.

\bibitem[Cai and Low(2006)]{cai2006adaptive}
T~Tony Cai and Mark~G Low.
\newblock Adaptive confidence balls.
\newblock \emph{The Annals of Statistics}, 34\penalty0 (1):\penalty0 202--228,
  2006.

\bibitem[Cai et~al.(2014)Cai, Low, and Ma]{cai2014adaptive}
T~Tony Cai, Mark Low, and Zongming Ma.
\newblock Adaptive confidence bands for nonparametric regression functions.
\newblock \emph{Journal of the American Statistical Association}, 109\penalty0
  (507):\penalty0 1054--1070, 2014.

\bibitem[Carpentier(2015)]{carpentier2015testing}
Alexandra Carpentier.
\newblock Testing the regularity of a smooth signal.
\newblock \emph{Bernoulli}, 21\penalty0 (1):\penalty0 465--488, 2015.

\bibitem[Genovese and Wasserman(2008)]{genovese2008adaptive}
Christopher Genovese and Larry Wasserman.
\newblock Adaptive confidence bands.
\newblock \emph{The Annals of Statistics}, 36\penalty0 (2):\penalty0 875--905,
  2008.

\bibitem[Gin{\'e} and Nickl(2010)]{gine2010confidence}
Evarist Gin{\'e} and Richard Nickl.
\newblock Confidence bands in density estimation.
\newblock \emph{The Annals of Statistics}, 38\penalty0 (2):\penalty0
  1122--1170, 2010.

\bibitem[Gin{\'e} and Nickl(2016)]{gine2016mathematical}
Evarist Gin{\'e} and Richard Nickl.
\newblock \emph{Mathematical foundations of infinite-dimensional statistical
  models}, volume~40.
\newblock Cambridge University Press, 2016.

\bibitem[Gupta et~al.(2020)Gupta, Podkopaev, and Ramdas]{gupta2020distribution}
Chirag Gupta, Aleksandr Podkopaev, and Aaditya Ramdas.
\newblock Distribution-free binary classification: prediction sets, confidence
  intervals and calibration.
\newblock \emph{arXiv preprint arXiv:2006.10564}, 2020.

\bibitem[Gy{\"o}rfi et~al.(2006)Gy{\"o}rfi, Kohler, Krzyzak, and
  Walk]{gyorfi2006distribution}
L{\'a}szl{\'o} Gy{\"o}rfi, Michael Kohler, Adam Krzyzak, and Harro Walk.
\newblock \emph{A distribution-free theory of nonparametric regression}.
\newblock Springer Science \& Business Media, 2006.

\bibitem[Hall and Horowitz(2013)]{hall2013simple}
Peter Hall and Joel Horowitz.
\newblock A simple bootstrap method for constructing nonparametric confidence
  bands for functions.
\newblock \emph{The Annals of Statistics}, 41\penalty0 (4):\penalty0
  1892--1921, 2013.

\bibitem[Hoffmann and Nickl(2011)]{hoffmann2011adaptive}
Marc Hoffmann and Richard Nickl.
\newblock On adaptive inference and confidence bands.
\newblock \emph{The Annals of Statistics}, 39\penalty0 (5):\penalty0
  2383--2409, 2011.

\bibitem[Jiang(2019)]{jiang2019non}
Heinrich Jiang.
\newblock Non-asymptotic uniform rates of consistency for k-{NN} regression.
\newblock In \emph{Proceedings of the AAAI Conference on Artificial
  Intelligence}, volume~33, pages 3999--4006, 2019.

\bibitem[Lei(2014)]{lei2014classification}
Jing Lei.
\newblock Classification with confidence.
\newblock \emph{Biometrika}, 101\penalty0 (4):\penalty0 755--769, 2014.

\bibitem[Lei and Wasserman(2014)]{lei2014distribution}
Jing Lei and Larry Wasserman.
\newblock Distribution-free prediction bands for non-parametric regression.
\newblock \emph{Journal of the Royal Statistical Society: Series B (Statistical
  Methodology)}, 76\penalty0 (1):\penalty0 71--96, 2014.

\bibitem[Lei et~al.(2018)Lei, G'Sell, Rinaldo, Tibshirani, and
  Wasserman]{lei2018distribution}
Jing Lei, Max G'Sell, Alessandro Rinaldo, Ryan~J Tibshirani, and Larry
  Wasserman.
\newblock Distribution-free predictive inference for regression.
\newblock \emph{Journal of the American Statistical Association}, 113\penalty0
  (523):\penalty0 1094--1111, 2018.

\bibitem[Li(1989)]{li1989honest}
Ker-Chau Li.
\newblock Honest confidence regions for nonparametric regression.
\newblock \emph{The Annals of Statistics}, 17\penalty0 (3):\penalty0
  1001--1008, 1989.

\bibitem[Low(1997)]{low1997nonparametric}
Mark~G Low.
\newblock On nonparametric confidence intervals.
\newblock \emph{The Annals of Statistics}, 25\penalty0 (6):\penalty0
  2547--2554, 1997.

\bibitem[McDiarmid(1998)]{mcdiarmid1998concentration}
Colin McDiarmid.
\newblock Concentration.
\newblock In \emph{Probabilistic Methods for Algorithmic Discrete Mathematics},
  pages 195--248. Springer, 1998.

\bibitem[Negahban et~al.(2012)Negahban, Ravikumar, Wainwright, and
  Yu]{negahban2012unified}
Sahand~N Negahban, Pradeep Ravikumar, Martin~J Wainwright, and Bin Yu.
\newblock A unified framework for high-dimensional analysis of ${M}$-estimators
  with decomposable regularizers.
\newblock \emph{Statistical Science}, 27\penalty0 (4):\penalty0 538--557, 2012.

\bibitem[Nouretdinov et~al.(2001)Nouretdinov, Vovk, Vyugin, and
  Gammerman]{nouretdinov2001pattern}
Ilia Nouretdinov, Volodya Vovk, Michael Vyugin, and Alex Gammerman.
\newblock Pattern recognition and density estimation under the general iid
  assumption.
\newblock In \emph{International Conference on Computational Learning Theory},
  pages 337--353. Springer, 2001.

\bibitem[Papadopoulos(2008)]{papadopoulos2008inductive}
Harris Papadopoulos.
\newblock Inductive conformal prediction: Theory and application to neural
  networks.
\newblock In \emph{Tools in Artificial Intelligence}. InTech, 2008.

\bibitem[Papadopoulos et~al.(2002)Papadopoulos, Proedrou, Vovk, and
  Gammerman]{papadopoulos2002inductive}
Harris Papadopoulos, Kostas Proedrou, Volodya Vovk, and Alex Gammerman.
\newblock Inductive confidence machines for regression.
\newblock In \emph{European Conference on Machine Learning}, pages 345--356.
  Springer, 2002.

\bibitem[Picard and Tribouley(2000)]{picard2000adaptive}
Dominique Picard and Karine Tribouley.
\newblock Adaptive confidence interval for pointwise curve estimation.
\newblock \emph{Annals of Statistics}, pages 298--335, 2000.

\bibitem[Sadinle et~al.(2019)Sadinle, Lei, and Wasserman]{sadinle2019least}
Mauricio Sadinle, Jing Lei, and Larry Wasserman.
\newblock Least ambiguous set-valued classifiers with bounded error levels.
\newblock \emph{Journal of the American Statistical Association}, 114\penalty0
  (525):\penalty0 223--234, 2019.

\bibitem[Szab{\'o} et~al.(2015)Szab{\'o}, Van Der~Vaart, and van
  Zanten]{szabo2015frequentist}
Botond Szab{\'o}, Aad~W Van Der~Vaart, and JH~van Zanten.
\newblock Frequentist coverage of adaptive nonparametric {B}ayesian credible
  sets.
\newblock \emph{The Annals of Statistics}, 43\penalty0 (4):\penalty0
  1391--1428, 2015.

\bibitem[Vovk(2012)]{vovk2012conditional}
Vladimir Vovk.
\newblock Conditional validity of inductive conformal predictors.
\newblock In \emph{Asian Conference on Machine Learning}, pages 475--490, 2012.

\bibitem[Vovk and Petej(2014)]{vovk2014venn}
Vladimir Vovk and Ivan Petej.
\newblock Venn--{A}bers predictors.
\newblock In \emph{Proceedings of the Thirtieth Conference on Uncertainty in
  Artificial Intelligence}, pages 829--838, 2014.

\bibitem[Vovk et~al.(2005)Vovk, Gammerman, and Shafer]{vovk2005algorithmic}
Vladimir Vovk, Alex Gammerman, and Glenn Shafer.
\newblock \emph{Algorithmic learning in a random world}.
\newblock Springer Science \& Business Media, 2005.

\bibitem[Wahba(1983)]{wahba1983bayesian}
Grace Wahba.
\newblock Bayesian ``confidence intervals'' for the cross-validated smoothing
  spline.
\newblock \emph{Journal of the Royal Statistical Society: Series B
  (Methodological)}, 45\penalty0 (1):\penalty0 133--150, 1983.

\end{thebibliography}

\appendix

\section{Additional proofs: lower bounds}

\subsection{Proof of Lemma~\ref{keylemma}}
The proof of this lemma follows a similar strategy as in \citet[Lemma 3]{barber2019limits},
and is a generalization of the construction used in the proof of  \citet[Proposition 5.1]{vovk2005algorithmic}.
First we embed the variables in the lemma into a distribution on triples $(X,Y,Z)$.
Writing $\tilde{P}_{X,Z}$ to denote the joint distribution of $(X_{n+1},Z_{n+1})$, we
define $\tilde{P}$ as follows:
\[\begin{cases}
(X,Z)\sim \tilde{P}_{X,Z},\\
Y|X,Z\sim \textnormal{Bernoulli}(Z).\end{cases}\]
Note that, marginalizing over $Z$, the pair $(X,Y)$ follows distribution $P$.
In other words, the joint distribution of 
\[(X_1,Y_1),\dots,(X_n,Y_n),X_{n+1},Z_{n+1}\]
under the  model $(X_i,Y_i,Z_i)\iidsim\tilde{P}$, matches the distribution specified in the lemma. 
Therefore, it is equivalent to prove the bound
\begin{equation}\label{eqn:keylemma_equiv}
\Pp{(X_i,Y_i,Z_i)\iidsim\tilde{P}}{Z_{n+1}\in\Ch_n(X_{n+1})} \geq 1-\alpha.\end{equation}

Next fix any integer $M\geq n+1$, and let $(X^{(m)},Z^{(m)})\iidsim \tilde{P}_{X,Z}$ for $m=1,\dots,M$.
Let $\Lcal$ specify this sequence of $M$ pairs.
Next, fixing $\Lcal$, we  draw $\{(X_i,Y_i,Z_i)\}_{i=1,\dots,n+1}$ as follows:
\begin{equation}\label{eqn:sample_without_replacement}
\begin{cases}\textnormal{Sample  $m_1,\dots,m_{n+1}$ uniformly {\bf without} replacement from $\{1,\dots,M\}$},\\
\textnormal{Set $(X_i,Z_i) = (X^{(m_i)},Z^{(m_i)})$ for each $i=1,\dots,n+1$},\\
\textnormal{Draw $Y_i\sim\textnormal{Bernoulli}(Z_i)$ for  each $i=1,\dots,n+1$}.\end{cases}\end{equation}
Then clearly, after marginalizing over $\Lcal$, the triples $(X_i,Y_i,Z_i)$ are drawn i.i.d.~from $\tilde{P}$. In other words,
\begin{equation}\label{eqn:keylemma_step1}\Pp{(X_i,Y_i,Z_i)\iidsim\tilde{P}}{Z_{n+1}\in\Ch_n(X_{n+1})} = \Ep{\Lcal}{\Ppst{\textnormal{Distrib.~\eqref{eqn:sample_without_replacement}}}{Z_{n+1}\in\Ch_n(X_{n+1})}{\Lcal}}.\end{equation}

Now consider any sequence $\Lcal = \{(X^{(m)},Z^{(m)})\}_{m=1,\dots,M}$, and let $Q_{\Lcal}$ be the distribution on $(X,Y,Z)$ defined by sampling $(X,Z)$ uniformly at random from $\Lcal$, 
then drawing $Y\sim\textnormal{Bernoulli}(Z)$.
We can consider drawing $(X_i,Y_i,Z_i)\iidsim Q_{\Lcal}$ for $i=1,\dots,n+1$, which is equivalent to
\begin{equation}\label{eqn:sample_with_replacement}
\begin{cases}\textnormal{Sample  $m_1,\dots,m_{n+1}$ uniformly {\bf with} replacement from $\{1,\dots,M\}$},\\
\textnormal{Set $(X_i,Z_i) = (X^{(m_i)},Z^{(m_i)})$ for each $i=1,\dots,n+1$},\\
\textnormal{Draw $Y_i\sim\textnormal{Bernoulli}(Z_i)$ for  each $i=1,\dots,n+1$}.\end{cases}\end{equation}
A simple calculation shows that, when 
 $m_1,\dots,m_{n+1}$ are sampled uniformly with replacement from $\{1,\dots,M\}$,
 the probability of the event $\{m_i=m_j\textnormal{ for any $i\neq j$}\}$ is bounded by $n^2/M$. Therefore,
for any fixed $\Lcal$, the total variation distance between the two sampling schemes~\eqref{eqn:sample_without_replacement}
and~\eqref{eqn:sample_with_replacement} is at most $n^2/M$, and so
\begin{multline}\label{eqn:keylemma_step2}
\Ppst{\textnormal{Distrib.~\eqref{eqn:sample_without_replacement}}}{Z_{n+1}\in\Ch_n(X_{n+1})}{\Lcal} \\\geq \Ppst{(X_i,Y_i,Z_i)\iidsim Q_{\Lcal}}{Z_{n+1}\in\Ch_n(X_{n+1})}{\Lcal}-  \frac{n^2}{M}.\end{multline}

Next we calculate $\pi_{Q_{\Lcal}}(x)$, i.e., the label probability under $Q_{\Lcal}$.
If $X^{(1)},\dots,X^{(M)}$ are all distinct for this $\Lcal$, then by definition of $Q_{\Lcal}$, we can see that $\pi_{Q_{\Lcal}}(X^{(m)})=Z^{(m)}$ for each $m=1,\dots,M$.
\edit{
(In particular, we can observe that the probability function $\pi_{Q_{\Lcal}}$ corresponding to this distribution $Q_{\Lcal}$ may in general be highly nonsmooth
even if $\pi_P$ is smooth---we have $|\pi_{Q_{\Lcal}}(X^{(m)}) -\pi_{Q_{\Lcal}}(X^{(m')})| = |Z^{(m)} - Z^{(m')}|$, which may be $\mathcal{O}(1)$ 
even if $|X^{(m)}-X^{(m')}|$ is arbitrarily small.)
Therefore,}
\begin{multline*} \Ppst{(X_i,Y_i,Z_i)\iidsim Q_{\Lcal}}{Z_{n+1}\in\Ch_n(X_{n+1})}{\Lcal}\\ =  \Ppst{(X_i,Y_i,Z_i)\iidsim Q_{\Lcal}}{\pi_{Q_{\Lcal}}(X_{n+1})\in\Ch_n(X_{n+1})}{\Lcal}
\geq 1-\alpha,\end{multline*}
where the last step holds since $\Ch_n$ is assumed to
be an algorithm that provides a $(1-\alpha)$-distribution-free confidence interval for any distribution~\eqref{eqn:distr_free_confidence_binary},
and in particular must provide coverage under $Q_{\Lcal}$.
In other words, for any $\Lcal$, we have proved that
\[ \Ppst{(X_i,Y_i,Z_i)\iidsim Q_{\Lcal}}{Z_{n+1}\in\Ch_n(X_{n+1})}{\Lcal} \geq (1-\alpha)\cdot\One{\textnormal{$X^{(1)},\dots,X^{(M)}$ are distinct}}.\]
Combining this bound with~\eqref{eqn:keylemma_step1} and~\eqref{eqn:keylemma_step2} establishes that
\[\Pp{(X_i,Y_i,Z_i)\iidsim\tilde{P}}{Z_{n+1}\in\Ch_n(X_{n+1})} \geq (1-\alpha)\cdot \PP{\textnormal{$X^{(1)},\dots,X^{(M)}$ are distinct}} - \frac{n^2}{M}.\]
Since $P$ is assumed to be nonatomic, the marginal $P_X$ is nonatomic as well, and so the $X^{(m)}$'s are distinct with probability 1. Therefore,
\[\Pp{(X_i,Y_i,Z_i)\iidsim\tilde{P}}{Z_{n+1}\in\Ch_n(X_{n+1})} \geq 1-\alpha - \frac{n^2}{M}.\]
Finally, since $M$ can be taken to be arbitrarily large, this establishes the desired bound~\eqref{eqn:keylemma_equiv}, and thus completes the proof of the lemma.

\subsection{Proof of Lemma~\ref{lem:ell_is_inf}}
We will prove a stronger form of Lemma~\ref{lem:ell_is_inf}. We define a set of distributions on $[0,1]$, $\Qcal = \Qcal^{(0)} \cup\Qcal^{(1)}$, 
where each distribution is a mixture of a point mass and a uniform distribution:
\begin{equation}\label{eqn:Qcal}
\begin{array}{ll}\Qcal^{(0)} = \{p\delta_0 + (1-p)\textnormal{Unif}[0,c] : p,c\in[0,1]\},\\
\Qcal^{(1)} = \{p\delta_1 + (1-p)\textnormal{Unif}[c,1] : p,c\in[0,1]\}.\end{array}
\end{equation}
\begin{lemma}\label{lem:ell_is_inf_strong}For any $t,a\in[0,1]$, define
\[\Fcal^+_{t,a} = \left\{\textnormal{\begin{tabular}{c}Measurable functions $f:[0,1]\rightarrow[0,1]$ satisfying\\ $\EE{f(Z)}\geq 1-a$ for any 
 random variable $Z\in[0,1]$ such that \\ either $0\leq \EE{Z}\leq t \leq \frac{1}{2}$ or $\frac{1}{2}\leq t\leq \EE{Z}\leq 1$\end{tabular}}\right\}\]
 and
 \[\Fcal^*_{t,a} = \left\{\textnormal{\begin{tabular}{c}Measurable functions $f:[0,1]\rightarrow[0,1]$ satisfying\\ $\Ep{Q}{f(Z)}\geq 1-a$ for any 
distribution $Q\in\Qcal$ with $\Ep{Q}{Z}=t$\end{tabular}}\right\},\]
and let $\Fcal_{t,a}$ be defined as in the statement of Lemma~\ref{lem:ell_is_inf}.
 Then it holds that
\[\ell(t,a) = \inf_{f\in\Fcal^+_{t,a}}\left\{\int_{s=0}^1 f(s)\;\mathsf{d}s \right\}  =  \inf_{f\in\Fcal_{t,a}}\left\{\int_{s=0}^1 f(s)\;\mathsf{d}s \right\} =  \inf_{f\in\Fcal^*_{t,a}}\left\{\int_{s=0}^1 f(s)\;\mathsf{d}s \right\}.\]
\end{lemma}

\begin{proof}[Proof of Lemma~\ref{lem:ell_is_inf_strong}]
Since $\Fcal^+_{t,a}\subseteq \Fcal_{t,a}\subseteq \Fcal^*_{t,a}$, it clearly holds that
\[ \inf_{f\in\Fcal^+_{t,a}}\left\{\int_{s=0}^1 f(s)\;\mathsf{d}s \right\} \geq  \inf_{f\in\Fcal_{t,a}}\left\{\int_{s=0}^1 f(s)\;\mathsf{d}s \right\} \geq \inf_{f\in\Fcal^*_{t,a}}\left\{\int_{s=0}^1 f(s)\;\mathsf{d}s \right\}.\] 
We now need to prove two remaining inequalities to establish the lemma:
\begin{equation}\label{eqn:ineq1_lem:ell_is_inf_strong}
\inf_{f\in\Fcal^*_{t,a}}\left\{\int_{s=0}^1 f(s)\;\mathsf{d}s \right\} \geq \ell(t,a)
\end{equation}
and
\begin{equation}\label{eqn:ineq2_lem:ell_is_inf_strong}
\inf_{f\in\Fcal^+_{t,a}}\left\{\int_{s=0}^1 f(s)\;\mathsf{d}s \right\} \leq \ell(t,a).
\end{equation}

First we prove~\eqref{eqn:ineq1_lem:ell_is_inf_strong}. Fix any $t,a\in[0,1]$ and any $f\in\Fcal^*_{t,a}$.
We split into cases:
\begin{itemize}
\item If $t=0$, then $\ell(t,a)=0$, and the bound holds trivially.
\item If $0<t\leq \frac{1}{2}\leq a$, let $Q=\textnormal{Unif}[0,2t]$. Then $Q\in\Qcal$ with  $\Ep{Q}{Z}=t$, and so we have
\[
\int_{s=0}^1 f(s)\;\mathsf{d}s
\geq 2t\cdot\int_{s=0}^{2t}   \frac{1}{2t}\cdot f(s)\;\mathsf{d}s
= 2t\cdot\Ep{Q}{f(Z)} \geq 2t(1-a)=\ell(t,a).
\]
\item If $0<t\leq a <\frac{1}{2}$, let $Q= (1-2a)\cdot \delta_0 + 2a\cdot\textnormal{Unif}[0,\frac{t}{a}]$.
Then $Q\in\Qcal$ with  $\Ep{Q}{Z}=t$, and so we have
\begin{align*}
\int_{s=0}^1 f(s)\;\mathsf{d}s
&\geq \frac{t}{2a^2}\cdot 2a\cdot \int_{s=0}^{t/a}\frac{a}{t} \cdot f(s)\;\mathsf{d}s\\
&\geq\frac{t}{2a^2}\cdot \left(2a\cdot \int_{s=0}^{t/a}\frac{a}{t} \cdot f(s)\;\mathsf{d}s + (1-2a) \cdot f(0) - (1-2a)\right)\\
&=\frac{t}{2a^2}\left(\Ep{Q}{f(Z)} - (1-2a)\right) \\
&\geq \frac{t}{2a^2}\left((1-a)-(1-2a)\right) 
=\frac{t}{2a}=\ell(t,a).
\end{align*}
\item If $0\leq a < t\leq \frac{1}{2}$, let $Q= (1-2t)\cdot \delta_0 + 2t\cdot\textnormal{Unif}[0,1]$.  
Then $Q\in\Qcal$ with  $\Ep{Q}{Z}=t$, and so we have
\begin{align*}
\int_{s=0}^1 f(s)\;\mathsf{d}s
&=\frac{1}{2t}\cdot 2t\cdot \int_{s=0}^1 f(s)\;\mathsf{d}s\\
&\geq \frac{1}{2t}\cdot \left(2t\cdot \int_{s=0}^1 f(s)\;\mathsf{d}s + (1-2t)\cdot f(0) - (1-2t)\right)\\
&= \frac{1}{2t}\cdot \left(\Ep{Q}{f(Z)} - (1-2t)\right)\\
&\geq \frac{1}{2t}\cdot \left((1-a) - (1-2t)\right)
= 1 - \frac{a}{2t} = \ell(t,a).
\end{align*}
\item By symmetry, the analogous calculations hold if $t>\frac{1}{2}$. 
\end{itemize}
Therefore $\int_{s=0}^1 f(s)\;\mathsf{d}s\geq \ell(t,a)$ in all cases, which completes the proof of~\eqref{eqn:ineq1_lem:ell_is_inf_strong}.

Next we turn to~\eqref{eqn:ineq2_lem:ell_is_inf_strong}.
For any $t,a\in[0,1]$, define the function $f_{t,a}:[0,1]\rightarrow [0,1]$ as in~\eqref{eqn:f_t_a}.
We first verify that \begin{equation}\label{eqn:f_t_a_length}\int_{s=0}^1f_{t,a}(s)\;\mathsf{d}s= \ell(t,a).\end{equation}
We split into cases:
\begin{itemize}
\item If $t=0$ then $\int_{s=0}^1 f_{t,a}(s)\;\mathsf{d}s=0 = \ell(t,a)$.
\item If $t=\frac{1}{2}$ then $\int_{s=0}^1 f_{t,a}(s)\;\mathsf{d}s=1-a = \ell(t,a)$.
\item If $0 < t < \frac{1}{2} \leq a$, then
\[\int_{s=0}^1 f_{t,a}(s)\;\mathsf{d}s = 2(1-a)\int_{s=0}^{2t} 1- \frac{s}{2t}\;\mathsf{d}s = 2(1-a)t = \ell(t,a).\]
\item If $0< t \leq a < \frac{1}{2}$, then
\[\int_{s=0}^1 f_{t,a}(s)\;\mathsf{d}s = \int_{s=0}^{t/a} 1 - \frac{as}{t}\;\mathsf{d}s = \frac{t}{2a} = \ell(t,a).\]
\item If $0\leq  a < t \leq \frac{1}{2}$, then
\[\int_{s=0}^1 f_{t,a}(s)\;\mathsf{d}s = \int_{s=0}^1 1- \frac{as}{t} \;\mathsf{d}s = 1-\frac{a}{2t} = \ell(t,a).\]
\item By symmetry, the analogous calculations hold if $t>\frac{1}{2}$. \end{itemize}
Therefore we have established that~\eqref{eqn:f_t_a_length} holds in all cases.
Next we check that $f_{t,a}\in\Fcal^+_{t,a}$. Let $Z\in[0,1]$ be any  random variable
satisfying either $0\leq \EE{Z}\leq t \leq \frac{1}{2}$ or $\frac{1}{2}\leq t\leq \EE{Z}\leq 1$. We again split into cases. 
\begin{itemize}
\item If $t=0$ then $\EE{Z}=0$ and so $Z=0$ with probability 1. Then $\EE{f_{t,a}(Z)} = f_{0,a}(0) = 1-a$.
\item If $t=\frac{1}{2}$ then $f_{t,a}(Z) = f_{\frac{1}{2},a}(Z) = 1-a$ almost surely, and so $\EE{f_{t,a}(Z)} =1-a$.
\item If $0 <   t < \frac{1}{2}$ and $a\geq \frac{1}{2}$,  then $\EE{Z}\leq t$ and so
\[
\EE{f_{t,a}(Z)}=2(1-a)\cdot\EE{\max\left\{1-\frac{Z}{2t},0\right\}}
\geq 
2(1-a)\cdot \left(1 - \frac{\EE{Z}}{2t}\right) \geq 1-a.\]
\item If $0 <   t < \frac{1}{2}$ and $a< \frac{1}{2}$, then $\EE{Z}\leq t$ and so
\[\EE{f_{t,a}(Z)} = \EE{\max\left\{1-\frac{aZ}{t},0\right\}}\geq 1-\frac{a\EE{Z}}{t} \geq 1-a.\]
\item By symmetry, the analogous calculations hold if $t>\frac{1}{2}$. \end{itemize}
Therefore, the bound $\EE{f_{t,a}(Z)}\geq 1-a$ holds in all cases, and so $f_{t,a}\in\Fcal^+_{t,a}$
by definition. Therefore, \[\inf_{f\in\Fcal^+_{t,a}} \left\{ \int_{s=0}^1f(s)\;\mathsf{d}s\right\}\leq \int_{s=0}^1f_{t,a}(s)\;\mathsf{d}s= \ell(t,a),\] and
so~\eqref{eqn:ineq2_lem:ell_is_inf_strong} holds.
\end{proof}

\subsection{Proof of Theorem~\ref{thm:lowerbd}}
Recall the set of distributions $\Qcal$ defined in~\eqref{eqn:Qcal}. Define a function
$g:[0,1]\times[0,1]\rightarrow [0,1]$
as
\[g(t,z) = \PPst{z\not\in\Ch_n(X_{n+1})}{\pi_P(X_{n+1})=t},\]
and define another function
$h:[0,1]\times\Qcal\rightarrow [0,1]$
as
\[h(t,Q) = \Ep{Q}{g(t,Z)}.\]
For each $t\in[0,1]$, let
$\Qcal_t = \left\{Q\in\Qcal : \Ep{Q}{Z} = t\right\}$.
By \citet[Theorem 18.19]{aliprantis2006infinite}, the function 
\[t\mapsto a(t) := \sup_{Q\in\Qcal_t}h(t,Q)\]
is measurable, and furthermore there exists a 
measurable function $t\mapsto Q_t\in\Qcal_t$ such that
\[ h(t,Q_t)  = a(t) =  \sup_{Q\in\Qcal_t}h(t,Q)\]
for all $t\in[0,1]$, as long as we verify the following conditions:
\begin{itemize}
\item $\Qcal$ is a separable metrizable space. To verify this condition, we will use the  
 total variation distance as our metric on $\Qcal$. 
 Define
\[\begin{array}{l}\Qcal_*^{(0)} = \{p\delta_0 + (1-p)\textnormal{Unif}[0,c] : p,c\in[0,1]\cap\mathbb{Q}\},\\
\Qcal_*^{(1)} = \{p\delta_1 + (1-p)\textnormal{Unif}[c,1] : p,c\in[0,1]\cap\mathbb{Q}\}.\end{array}\]
where $\mathbb{Q}$ is the set of rational numbers. Then $ \Qcal_*^{(0)}\cup \Qcal_*^{(1)}$ is a countable set, and is a dense subset of $\Qcal$
(under the total variation distance).
Therefore, $\Qcal$ is separable. 
\item $\Qcal_t$ is compact for all $t$. 
To prove this, first consider
the case $t\leq \frac{1}{2}$. 
If $t=0$ then $\Qcal_t\cap\Qcal^{(0)}= \{\delta_0\}$, and is trivially compact.
Now consider $0<t\leq \frac{1}{2}$. Writing $Q_{p,c}=p\delta_0 + (1-p)\textnormal{Unif}[0,c]$,  fix any $Q_{p,c},Q_{p',c'}\in\Qcal_t\cap\Qcal^{(0)}$.
We must have $c,c' \geq 2t$ and $p,p'\leq1-2t$ in order to obtain expected value $t$, and we can calculate that this implies
\[ |p-p'| + 2t|c-c'|\leq \dtv(Q_{p,c},Q_{p',c'}) \leq |p-p'| + (2t)^{-1}|c-c'|.\]
This proves that, on $\Qcal_t\cap \Qcal^{(0)}$, the topology induced by total variation distance is the same as the topology induced by the Euclidean distance on $(p,c)\in[0,1]^2$.
Therefore, since $\Qcal_t\cap \Qcal^{(0)}= \{Q_{p,c}:(1-p)\cdot\frac{c}{2} = t\}$ corresponds to a closed subset of $(p,c)\in[0,1]^2$, this set is compact
for the case $t\leq \frac{1}{2}$. If instead $t>\frac{1}{2}$, then $\Qcal_t\cap \Qcal^{(0)}=\emptyset$ and is trivially compact.
An analogous argument shows that $\Qcal_t\cap \Qcal^{(1)}$ is compact. Therefore, $\Qcal_t = (\Qcal_t\cap \Qcal^{(0)})\cup (\Qcal_t\cap \Qcal^{(1)})$ is compact.
\item $t\mapsto h(t,Q)$ is measurable for all $Q$, which holds since $t\mapsto g(t,z)$ is measurable
by definition of the conditional expectation, and therefore $t\mapsto h(t,Q)=\Ep{Q}{g(t,Z)}$ is measurable as well.
\item $Q\mapsto h(t,Q)$ is continuous for all $t$, which holds since $\big|h(t,Q)-h(t,Q')\big|\leq \dtv(Q,Q')$ for all $Q,Q'$.
\end{itemize}

Now define a random variable $Z_{n+1}$ drawn
from the distribution $Q_{\pi_P(X_{n+1})}$ after conditioning on $X_{n+1}$. Then by definition, $Z_{n+1}$ satisfies
 \[Z_{n+1}\independent \Ch_n(X_{n+1}) \mid X_{n+1}\textnormal{\quad and \quad}
\EEst{Z_{n+1}}{X_{n+1}} = \pi_P(X_{n+1})\textnormal{ almost surely},\]
and by Lemma~\ref{keylemma} it therefore holds that
$\PP{Z_{n+1}\in\Ch_n(X_{n+1})}\geq 1-\alpha$.
Therefore,
\begin{align*}
\alpha
&\geq \PP{Z_{n+1}\not\in\Ch_n(X_{n+1})}\\
&= \EE{\EEst{\PPst{Z_{n+1}\not\in\Ch_n(X_{n+1})}{Z_{n+1},\pi_P(X_{n+1})}}{\pi_P(X_{n+1})}}\\
&= \EE{\EEst{g(\pi_P(X_{n+1}),Z_{n+1})}{\pi_P(X_{n+1})}}\\
&= \EE{h(\pi_P(X_{n+1}),Q_{\pi_P(X_{n+1})})}
=\EE{a(\pi_P(X_{n+1}))}
=\Ep{T\sim\Pi_P}{a(T)}.
\end{align*}
Recalling the definition of $L_\alpha(\Pi_P)$, this means that
$L_\alpha(\Pi_P)\leq \Ep{T\sim \Pi_P}{\ell(T,a(T))}$.

Next, fix any $t\in[0,1]$. By definition of $a(t)$, for any $Q\in\Qcal_t$,
$ \Ep{Q}{g(t,Z)}\leq  a(t)$.
In the notation of Lemma~\ref{lem:ell_is_inf_strong}, the function $z\mapsto 1-g(t,z)$ belongs to $\Fcal^*_{t,a(t)}$, and so
\[\int_{s=0}^1 (1-g(t,s))\;\mathsf{d}s\geq \ell(t,a(t)).\]
Since this holds for all $t\in[0,1]$, we therefore have
\[\Ep{T\sim \Pi_P}{\int_{s=0}^1 (1-g(T,s))\;\mathsf{d}s}\geq \Ep{T\sim \Pi_P}{\ell(T,a(T))}\geq L_\alpha(\Pi_P).\]
Finally, applying Fubini's theorem completes the proof:
\begin{align*}
\EE{\leb(\Ch_n(X_{n+1}))} 
&= \EE{\int_{s=0}^1\One{s\in\Ch_n(X_{n+1})}\;\mathsf{d}s}\\
&=\EE{\int_{s=0}^1 \PPst{s\in\Ch_n(X_{n+1})}{\pi_P(X_{n+1})}\;\mathsf{d}s}\\
&=\EE{\int_{s=0}^1 1 - g(\pi_P(X_{n+1}),s)\;\mathsf{d}s}\\
&=\Ep{T\sim \Pi_P}{\int_{s=0}^1 1 - g(T,s)\;\mathsf{d}s}
\geq L_\alpha(\Pi_P).
\end{align*}

\section{Additional proofs: upper bounds}
\subsection{Proof of Lemma~\ref{lem:generic_a_t}}

First we check expected length.
By definition of $C_{\bt,\ba}(X_{n+1})$, we have
\begin{align*}
\EE{\leb(C_{\bt,\ba}(X_{n+1}))} 
&= \sum_{m=1}^M p_{P,m}\EEst{\leb(C_{\bt,\ba}(X_{n+1}))}{X_{n+1}\in\Xcal_m}\\
&= \sum_{m=1}^M p_{P,m}\EE{\leb(\{s\in [0,1] : f_{t_m,a_m}(s) \geq U\})}\\
&= \sum_{m=1}^M p_{P,m}\EE{\int_{s=0}^1 \One{f_{t_m,a_m}(s) \geq U}\;\mathsf{d}s}\\
&= \sum_{m=1}^M p_{P,m}\int_{s=0}^1 \PP{f_{t_m,a_m}(s) \geq U}\;\mathsf{d}s\textnormal{\quad by Fubini's theorem}\\
&= \sum_{m=1}^M p_{P,m}\int_{s=0}^1 f_{t_m,a_m}(s) \;\mathsf{d}s
=\sum_{m=1}^m p_{P,m} \ell(t_m,a_m),
\end{align*}
where the last step applies the calculation~\eqref{eqn:f_t_a_length} from the proof of Lemma~\ref{lem:ell_is_inf_strong}.

Next we check coverage under assumption~\eqref{eqn:pi_vs_t}. We have
\begin{align*}
\PP{\pi_P(X_{n+1})\in C_{\bt,\ba}(X_{n+1})}
&=\PP{f_{t_m(X_{n+1}),a_m(X_{n+1})}\big(\pi_P(X_{n+1})\big) \geq U}\\
&=\EE{\PPst{f_{t_m(X_{n+1}),a_m(X_{n+1})}\big(\pi_P(X_{n+1})\big) \geq U}{X_{n+1}}}\\
&=\EE{f_{t_m(X_{n+1}),a_m(X_{n+1})}\big(\pi_P(X_{n+1})\big)}\\
&=\sum_{m=1}^M p_{P,m}\EEst{f_{t_m,a_m}\big(\pi_P(X_{n+1})\big)}{X_{n+1}\in\Xcal_m}.
\end{align*}
Next, for each $m$, in the proof of Lemma~\ref{lem:ell_is_inf_strong} 
we established that $f_{t_m,a_m}\in\Fcal^+_{t_m,a_m}$. By definition of this set, this implies that
\[\EEst{f_{t_m,a_m}\big(\pi_P(X_{n+1})\big)}{X_{n+1}\in\Xcal_m} \geq 1-a_m,\]
since $\EEst{\pi_P(X_{n+1})}{X_{n+1}\in\Xcal_m} = \pi_{P,m}$, and $\pi_{P,m}$ satisfies the assumption~\eqref{eqn:pi_vs_t}.
Therefore,
\[\PP{\pi_P(X_{n+1})\in C_{\bt,\ba}(X_{n+1})}
\geq\sum_{m=1}^M p_{P,m}(1-a_m) = 1 - \sum_{m=1}^M p_{P,m}a_m,\]
as desired.

\subsection{Proof of Lemma~\ref{lem:upperbd_oracle}}
First, the coverage statement follows immediately from Lemma~\ref{lem:generic_a_t}, since $ \sum_{m=1}^M p_{P,m} a^*_{P,m} \leq \alpha$ by definition
of $\ba^*_P$. Now we 
consider the expected length. 
For each $x$ define
\[\delta(x) = \big|\pi_P(x) - \pi_{P,m(x)}\big|.\]
Let
\begin{equation}\label{eqn:eps_Delta_bound}\eps = \Ep{P_X}{\sqrt{\frac{\delta(X)}{2\alpha}}}\leq   \sqrt{\frac{\Ep{P_X}{\delta(X)}}{2\alpha}}=\sqrt{\frac{\Delta_P(\Xcal_{1:M})}{2\alpha}} .\end{equation}
If $\eps>1$ then the result of the lemma holds trivially,  since $\leb(C^*_P(X))\leq\leb([0,1])=1$ always.
Therefore we can restrict our attention to the case that $\eps\in[0,1]$.

Now fix any function $a:[0,1]\rightarrow[0,1]$ with $\Ep{T\sim\Pi_P}{a(T)}\leq\alpha$. Define a vector $\ba^\circ\in[0,1]^M$ with entries
\[a^\circ_m= \min\left\{1, \Epst{P_X}{\sqrt{\frac{\alpha\delta(X)}{2}}+  (1-\eps)\cdot a(\pi_P(X))}{X\in\Xcal_m} \right\}.\]
We can calculate
\begin{align*}
\sum_{m=1}^M p_{P,m} a^\circ_m
&\leq\sum_{m=1}^M p_{P,m}\Epst{P_X}{\sqrt{\frac{\alpha\delta(X)}{2}}+  (1-\eps)\cdot a(\pi_P(X))}{X\in\Xcal_m}\\
&=\Ep{P_X}{\sqrt{\frac{\alpha\delta(X)}{2}}} + (1-\eps)\cdot \Ep{P_X}{a(\pi_P(X))}\\
&\leq\Ep{P_X}{\sqrt{\frac{\alpha\delta(X)}{2}}}  + (1-\eps)\cdot \alpha\textnormal{\quad by definition of $a$}\\
&=\alpha\textnormal{\quad by definition of $\eps$.}
\end{align*}
Therefore, $\ba^\circ$ is feasible for the optimization problem~\eqref{eqn:a_star}, and so by optimality of $\ba^*_P$, we must have
\[\sum_{m=1}^M p_{P,m} \ell(\pi_{P,m},a^*_{P,m})\leq \sum_{m=1}^M p_{P,m} \ell(\pi_{P,m},a^\circ_m).\]
We next need to bound this right-hand side.
For each $m$, we have either $a^\circ_m=1$ in which case $\ell(\pi_{P,m},a^\circ_m)  = 0$, or if instead $a^\circ_m<1$, then we have
\begin{align*}\ell(\pi_{P,m},a^\circ_m) 
&=\ell\left(\pi_{P,m}, \Epst{P_X}{\sqrt{\frac{\alpha\delta(X)}{2}}+  (1-\eps)\cdot a(\pi_P(X))}{X\in\Xcal_m}\right)\\
&\leq  \Epst{P_X}{\ell\left(\pi_{P,m}, \sqrt{\frac{\alpha\delta(X)}{2}} + (1-\eps)\cdot a(\pi_P(X))\right)}{X\in\Xcal_m}\\
&\leq  \Epst{P_X}{\ell\left(\pi_P(X),  (1-\eps)\cdot a(\pi_P(X))\right)+\frac{|\pi_P(X) - \pi_{P,m}|}{2 \sqrt{\frac{\alpha\delta(X)}{2}}}}{X\in\Xcal_m}\\
&=  \Epst{P_X}{\ell\left(\pi_P(X),  (1-\eps)\cdot a(\pi_P(X))\right)+\sqrt{\frac{\delta(X)}{2\alpha}}}{X\in\Xcal_m},
\end{align*}
where the second step applies the fact that $\ell(t,a)$ is convex in $a$,
while the third step applies the fact that $\ell(t,a)$ is $\frac{1}{2a}$-Lipschitz in $t$ for any fixed $a$, 
and is nonincreasing in $a$ for any fixed $t$.
Next, since $\ell(t,a)$ is convex in $a$ and is bounded by 1, we have
\begin{multline*}\ell\left(\pi_P(X),  (1-\eps)\cdot a(\pi_P(X))\right) = \ell\left(\pi_P(X),  (1-\eps)\cdot a(\pi_P(X)) + \eps\cdot 0\right)\\\leq (1-\eps)\cdot \ell(\pi_P(X),a(\pi_P(X))) + \eps\cdot\ell(\pi_P(X),0)
\leq \ell(\pi_P(X),a(\pi_P(X))) +\eps.\end{multline*}
Combining everything, then,
\[\ell(\pi_{P,m},a^\circ_m)  \leq \Epst{P_X}{ \ell(\pi_P(X),a(\pi_P(X)))}{X\in\Xcal_m} +\eps + \Epst{P_X}{\sqrt{\frac{\delta(X)}{2\alpha}}}{X\in\Xcal_m}.\]
Summing over $m$, we obtain
\begin{align*}
\sum_{m=1}^M p_{P,m} \ell(\pi_{P,m},a^*_{P,m})
&\leq \sum_{m=1}^M p_{P,m} \ell(\pi_{P,m},a^\circ_m)\\
&  \leq\Ep{P_X}{ \ell(\pi_P(X),a(\pi_P(X)))} + \eps + \Ep{P_X}{\sqrt{\frac{\delta(X)}{2\alpha}}}\\
&\leq \Ep{T\sim \Pi_P}{ \ell(T,a(T))}  + \sqrt{\frac{2\Delta_P(\Xcal_{1:M})}{\alpha}},\end{align*}
where the last step applies~\eqref{eqn:eps_Delta_bound} and the definition of $\Pi_P$.
Since we proved this bound for an arbitrary function $a:[0,1]\rightarrow[0,1]$ satisfying $\Ep{T\sim\Pi_P}{a(T)}\leq\alpha$, we have therefore shown that
\[\sum_{m=1}^M p_{P,m} \ell(\pi_{P,m},a^*_{P,m})\leq L_\alpha(\Pi_P) +\sqrt{\frac{2\Delta_P(\Xcal_{1:M})}{\alpha}}.\]
Finally,  Lemma~\ref{lem:generic_a_t} implies that $\Ep{P_X}{\leb(C^*_P(X))}=\sum_{m=1}^M p_{P,m} \ell(\pi_{P,m},a^*_{P,m})$, which completes the proof.

\subsection{Proof of Theorem~\ref{thm:upperbd}}
Before proving the theorem, we state two supporting lemmas. The first is a basic property of the function $\ell(t,a)$.
\begin{lemma}\label{lem:ell_sqrt_bound}
Fix any $a,t,t'\in [0,1]$. If it holds that
\[|t-t'| \leq 2\delta^2 + \delta\sqrt{8\min\{t,1-t\}},\]
then
\[\ell(t',a+r\delta) \leq \ell(t,a) + \frac{\delta}{r}\]
for any $r$ such that $a+r\delta\leq 1$.
\end{lemma}
\noindent The second is a simple consequence of the Chernoff bound on the Binomial distribution.
\begin{lemma}\label{lem:chernoff}
Let $n\geq 2$. Then under any distribution $P$, with probability at least $1-\frac{\alpha}{n}$, the following bounds hold for all $m=1,\dots,M$:
\begin{equation}\label{eqn:check_smoothing_p} p_{P,m}\leq\tilde p_m\left(1-\frac{1}{n}\right)\textnormal{ and }\tilde p_m \leq  p_{P,m} + \sqrt{p_{P,m}\cdot \frac{18\log(4Mn/\alpha)}{n}}+\frac{12\log(4Mn/\alpha)}{n},\end{equation}
and
\begin{multline}\label{eqn:check_smoothing_pi} |\tilde\pi_m-\pi_{P,m}|\leq \sqrt{\min\{\pi_{P,m},1-\pi_{P,m}\}\cdot \frac{18\log(4Mn/\alpha)}{n p_{P,m}}}+\frac{12\log(4Mn/\alpha)}{n p_{P,m}},\\\textnormal{ and either }
0\leq \pi_{P,m}\leq\tilde\pi_m\leq\frac{1}{2}\textnormal{\ or\  }\frac{1}{2}\leq\tilde\pi_m\leq\pi_{P,m}\leq 1.\end{multline}
\end{lemma}

Now we turn to the proof of the theorem. We first prove the coverage bound under an arbitrary distribution $P$.
Let $\Ecal_1$ be the event that, for all $m$, it holds that
\[ p_{P,m}\leq\tilde p_m\left(1-\frac{1}{n}\right),\textnormal{ and either }0\leq \pi_{P,m}\leq\tilde\pi_m\leq\frac{1}{2}\textnormal{\ or\  }\frac{1}{2}\leq\tilde\pi_m\leq\pi_{P,m}\leq 1.\]
 Lemma~\ref{lem:chernoff} verifies that $\PP{\Ecal_1}\geq 1-\frac{\alpha}{n}$
 for any distribution $P$.
Recall that $\Ch_n(X_{n+1})=C_{\tilde\bpi,\tilde\ba}(X_{n+1})$ depends on 
the training data only through $\tilde\bpi,\tilde\ba$, which themselves are functions of the $\widehat p_m$'s 
and $\widehat\pi_m$'s. 
By  Lemma~\ref{lem:generic_a_t},
on the event $\Ecal_1$, it holds that
\begin{multline*}
\PPst{\pi_P(X_{n+1})\not\in \Ch_n(X_{n+1})}{\{(X_i,Y_i)\}_{i=1,\dots,n}}\\
 \leq \sum_{m=1}^M p_{P,m}\tilde a_m
 \leq \left(1-\frac{1}{n}\right)\sum_{m=1}^M \tilde p_m\tilde a_m
\leq \alpha\left(1-\frac{1}{n}\right),
\end{multline*}
where the last step holds by definition of $\tilde\ba$. Therefore,
\[\PP{\pi_P(X_{n+1})\not\in \Ch_n(X_{n+1})} \leq \PPst{\pi_P(X_{n+1})\not\in \Ch_n(X_{n+1})}{\Ecal_1} + \frac{\alpha}{n} \leq \alpha,\]
 which verifies the distribution-free coverage guarantee.

Next we turn to proving the bound on expected length.
Applying Lemma~\ref{lem:generic_a_t}, it holds that
\[\EEst{\leb(\Ch_n(X_{n+1}))}{\{(X_i,Y_i)\}_{i=1,\dots,n}} = \sum_{m=1}^M p_{P,m}\ell(\tilde\pi_m,\tilde a_m).\]
Let $\mathcal{E}_2$ be the event that, for all $m$,~\eqref{eqn:check_smoothing_p} and~\eqref{eqn:check_smoothing_pi}
both hold.
By Lemma~\ref{lem:chernoff}, $\PP{\Ecal_2}\geq 1-\frac{\alpha}{n}$, and 
therefore,
\begin{equation}\label{eqn:length_Ecal_update1}
\EE{\leb(\Ch_n(X_{n+1}))}\leq \EE{\one{\Ecal_2}\cdot \sum_{m=1}^M p_{P,m}\ell(\tilde\pi_m,\tilde a_m)} +\frac{\alpha}{n}.\end{equation}
We therefore now need to bound $\sum_{m=1}^M p_{P,m}\ell(\tilde\pi_m,\tilde a_m)$ on the event $\Ecal_2$.

From this point on, all our calculations will be conditional on $\{(X_i,Y_i)\}_{i=1,\dots,n}$ and we will assume that $\Ecal_2$ holds.
For all $m$, define
\[\delta_m = \begin{cases} \frac{1}{\sqrt{2}}\left|\sqrt{\tilde\pi_m} - \sqrt{\pi_{P,m}}\right|, & \pi_{P,m}\leq \frac{1}{2},\\
\frac{1}{\sqrt{2}}\left|\sqrt{1-\tilde\pi_m} - \sqrt{1-\pi_{P,m}}\right|, & \pi_{P,m}> \frac{1}{2},\end{cases}
\]
and let
\[\eps = 9\sqrt{\frac{M\log(4Mn/\alpha)}{\alpha n}}.\]
By~\eqref{eqn:check_smoothing_pi}, we can calculate
\begin{equation}\label{eqn:delta_eps}\sum_{m=1}^M p_{P,m}\delta_m \leq \sum_{m=1}^M p_{P,m}\cdot\sqrt{\frac{6\log(4Mn/\alpha)}{n p_{P,m}}}
\leq \sqrt{\frac{6M\log(4Mn/\alpha)}{n }} = \frac{\eps\sqrt{6\alpha}}{9},\end{equation}
where the next-to-last step holds since $\sum_{m=1}^M \sqrt{p_{P,m}} \leq \sqrt{M}\sqrt{\sum_{m=1}^M p_{P,m}} = \sqrt{M}$.

Next, we will assume for now that $\eps\leq 1$. Let
\[a^\circ_m = \min\left\{ 1, (1-\eps) \cdot a^*_{P,m} +  \delta_m\sqrt{\alpha}\right\}\]
where $\ba^*_P$ is defined as in~\eqref{eqn:a_star}.
Now fix any $m$. If $a^\circ_m=1$ then $\ell(\tilde\pi_m,a^\circ_m)=0$. If not, then $a^\circ_m= (1-\eps) \cdot a^*_{P,m} +  \delta_m\sqrt{\alpha}$,
and we now derive a bound on  $\ell(\tilde\pi_m,a^\circ_m)$.
By definition of $\delta_m$, we have
\[|\tilde\pi_m-\pi_{P,m}|\leq\delta_m\cdot \sqrt{8\min\{\pi_{P,m},1-\pi_{P,m}\}}+2\delta_m^2.\]
Applying 
Lemma~\ref{lem:ell_sqrt_bound}, we have
\begin{align*}
\ell(\tilde\pi_m,a^\circ_m) 
&\leq \ell(\pi_{P,m},a^\circ_m - \delta_m\sqrt{\alpha}) + \frac{\delta_m}{\sqrt{\alpha}}\\
&= \ell(\pi_{P,m},(1-\eps) a^*_{P,m}) + \frac{\delta_m}{\sqrt{\alpha}}\\
&= \ell(\pi_{P,m},(1-\eps) a^*_{P,m} + \eps\cdot 0) + \frac{\delta_m}{\sqrt{\alpha}}\\
&\leq (1-\eps)\cdot  \ell(\pi_{P,m},a^*_{P,m}) + \eps\cdot\ell(\pi_{P,m}, 0) + \frac{\delta_m}{\sqrt{\alpha}}\\
&\leq \ell(\pi_{P,m},a^*_{P,m}) + \eps + \frac{\delta_m}{\sqrt{\alpha}},
\end{align*}
since $\ell(t,a)$ is convex in $a$, and bounded by $1$.
Summing over all $m$, and applying~\eqref{eqn:delta_eps},
\begin{multline}\label{eqn:length_Ecal_update2}\sum_{m=1}^M p_{P,m}\ell(\tilde\pi_m,a^\circ_m) \leq \sum_{m=1}^M p_{P,m} \ell(\pi_{P,m},a^*_{P,m})+ \left(1+\frac{\sqrt{6}}{9}\right)\eps\\
 \leq  L_\alpha(\Pi_P) + \sqrt{\frac{2\Delta_P(\Xcal_{1:M})}{\alpha}} + \left(1+\frac{\sqrt{6}}{9}\right)\eps,\end{multline}
where the last step applies Lemma~\ref{lem:upperbd_oracle}
along with the calculation $\Ep{P_X}{\leb(C^*_P(X))} = \sum_{m=1}^M p_{P,m} \ell(\pi_{P,m},a^*_{P,m})$ from the proof of that lemma. 

Next we will need to relate $\sum_{m=1}^M p_{P,m}\ell(\tilde\pi_m,a^\circ_m)$ to $\sum_{m=1}^M p_{P,m}\ell(\tilde\pi_m,\tilde a_m)$.
Recalling the definition of $\tilde\ba$ in~\eqref{eqn:a_tilde}, we see that by optimality of $\tilde\ba$,
\begin{equation}\label{eqn:length_Ecal_update3}\textnormal{If $\sum_{m=1}^M\tilde p_m a^\circ_m\leq \alpha$, then\  $\sum_{m=1}^M p_{P,m}\ell(\tilde\pi_m,\tilde a_m)\leq \sum_{m=1}^M p_{P,m}\ell(\tilde\pi_m,a^\circ_m)$.}\end{equation}
We now turn to verifying that $\sum_{m=1}^M\tilde p_m a^\circ_m\leq \alpha$, to ensure that $\ba^\circ$ is feasible for the optimization problem~\eqref{eqn:a_tilde}.
First, we have
\begin{align*}
\sum_{m=1}^M p_{P,m} a^\circ_m
&\leq \sum_{m=1}^M p_{P,m} \left((1-\eps) \cdot a^*_{P,m} +  \delta_m\sqrt{\alpha}\right)\\
&= (1-\eps)\sum_{m=1}^Mp_{P,m}a^*_{P,m} + \sqrt{\alpha}\sum_{m=1}^M p_{P,m}\delta_m\\
&\leq (1-\eps)\alpha + \sqrt{\alpha}\sum_{m=1}^M p_{P,m}\delta_m\textnormal{\quad by definition of $\ba^*_P$~\eqref{eqn:a_star}}\\
&\leq (1-\eps)\alpha + \sqrt{\alpha}\cdot \frac{\eps\sqrt{6\alpha}}{9}\textnormal{\quad by~\eqref{eqn:delta_eps}}\\
&= \alpha-\alpha\eps\left(1 - \frac{\sqrt{6}}{9}\right).
\end{align*}
Next, applying~\eqref{eqn:check_smoothing_p}, along with the fact that $a^\circ_m\leq 1$ by construction, we have
\begin{align*}
&\sum_{m=1}^M\tilde p_m a^\circ_m
\leq \sum_{m=1}^M \left( p_{P,m} + \sqrt{p_{P,m}\cdot \frac{18\log(4Mn/\alpha)}{n}}+\frac{12\log(4Mn/\alpha)}{n}\right) \cdot a^\circ_m\\
&\leq \sum_{m=1}^M p_{P,m} a^\circ_m+ \left(\sum_{m=1}^M\sqrt{p_{P,m} a^\circ_m}\right)\cdot \sqrt{\frac{18\log(4Mn/\alpha)}{n}}+\frac{12M\log(4Mn/\alpha)}{n}\\
&\leq \sum_{m=1}^M p_{P,m} a^\circ_m+ \sqrt{\sum_{m=1}^Mp_{P,m} a^\circ_m}\cdot \sqrt{\frac{18M\log(4Mn/\alpha)}{n}}+\frac{12M\log(4Mn/\alpha)}{n}\\
&\leq\alpha-\alpha\eps\left(1 - \frac{\sqrt{6}}{9}\right) + \sqrt{\alpha}\cdot \sqrt{\frac{18M\log(4Mn/\alpha)}{n}}+\frac{12M\log(4Mn/\alpha)}{n} \\
&\leq \alpha,
\end{align*}
where the last step plugs in the definition of $\eps$ along with the assumption that $\eps\leq 1$. Thus we have proved that $\sum_{m=1}^M\tilde p_m a^\circ_m\leq \alpha$.
This means that
$\ba^\circ$ is feasible for the optimization problem~\eqref{eqn:a_tilde}. Combining~\eqref{eqn:length_Ecal_update2} with~\eqref{eqn:length_Ecal_update3}, we have therefore proved that on the event $\Ecal_2$, 
\[\sum_{m=1}^M p_{P,m}\ell(\tilde\pi_m,\tilde a_m)
 \leq  L_\alpha(\Pi_P) + \sqrt{\frac{2\Delta_P(\Xcal_{1:M})}{\alpha}} +\left(1+\frac{\sqrt{6}}{9}\right)\eps.\]
After combining with~\eqref{eqn:length_Ecal_update1}, we therefore have
\[
\EE{\leb(\Ch_n(X_{n+1}))}\leq L_\alpha(\Pi_P) + \sqrt{\frac{2\Delta_P(\Xcal_{1:M})}{\alpha}} + \left(1+\frac{\sqrt{6}}{9}\right)\eps +\frac{\alpha}{n},\]
as long as $\eps\leq 1$. Furthermore, the assumption $\eps\leq 1$ ensures that $\log(4Mn/\alpha)\leq 2\log n$, and so plugging in the definition of $\eps$, we have proved that
\[
\EE{\leb(\Ch_n(X_{n+1}))}\leq L_\alpha(\Pi_P) + \sqrt{\frac{2\Delta_P(\Xcal_{1:M})}{\alpha}} + c'\sqrt{\frac{M\log n}{\alpha n}}\]
for a sufficiently large universal constant $c'$. If instead we have $\eps>1$, then we have
$\EE{\leb(\Ch_n(X_{n+1}))}\leq1 \leq c''\sqrt{\frac{M\log n}{\alpha n}} $
for a sufficiently large universal constant $c''$.
Taking $c = \max\{c',c''\}$, we have completed the proof of the theorem. 

\subsection{Proof of Corollary~\ref{cor:upperbd}}
To help with the proof, we begin by adapting our previous notation to the setting of a data-dependent partition.
The partition $\widehat{\Xcal}^{\Rcal}_{1:M}$ is a function of the first half of the data,
i.e., data points $\{(X_i,Y_i)\}_{i=1,\dots,\lfloor \frac{n}{2}\rfloor}$. Conditioning on these data points, we define
\[p_{P,m} = \Ppst{P_X}{X\in \widehat{\Xcal}^{\Rcal}_m}{\{(X_i,Y_i)\}_{i=1,\dots,\lfloor \frac{n}{2}\rfloor}}\]
and
\[\pi_{P,m} =   \Epst{P_X}{\pi_P(X)}{X\in\widehat{\Xcal}^{\Rcal}_m;\{(X_i,Y_i)\}_{i=1,\dots,\lfloor \frac{n}{2}\rfloor}},\]
where we should interpret this to mean that the partition $\widehat{\Xcal}^{\Rcal}_{1:M}$ is treated 
as fixed, and the probability and expectation are calculated with respect to an independent draw $X\sim P_X$.
Similarly, we will write
\[\Delta\big(\widehat{\Xcal}^{\Rcal}_{1:M}\big) 
 = \Epst{P_X}{|\pi_P(X) - \pi_{P,m(X)}|}{\{(X_i,Y_i)\}_{i=1,\dots,\lfloor \frac{n}{2}\rfloor}}.\]
These quantities are now functions of the data points $\{(X_i,Y_i)\}_{i=1,\dots,\lfloor \frac{n}{2}\rfloor}$.

Next we apply Theorem~\ref{thm:upperbd}. Specifically,
we will condition on the data points $\{(X_i,Y_i)\}_{i=1,\dots,\lfloor \frac{n}{2}\rfloor}$ used to choose the partition
so that the partition can be treated as fixed, and will
 apply Theorem~\ref{thm:upperbd} with $\lceil \frac{n}{2}\rceil\geq 2$
in place of $n$ (i.e., we apply the theorem to data points $i=\lfloor \frac{n}{2}\rfloor +1,\dots,n$ in place of $i=1,\dots,n$). 
This proves that
\[\Ppst{(X_i,Y_i)\iidsim P}{\pi_P(X_{n+1})\in \Ch^{\Rcal}_n(X_{n+1})}{\{(X_i,Y_i)\}_{i=1,\dots,\lfloor \frac{n}{2}\rfloor}} \geq 1- \alpha\]
and
\begin{multline*}\Epst{(X_i,Y_i)\iidsim P}{\leb(\Ch^{\Rcal}_n(X_{n+1}))}{\{(X_i,Y_i)\}_{i=1,\dots,\lfloor \frac{n}{2}\rfloor}}
\\ \leq L_\alpha(\Pi_P) + \sqrt{\frac{2\Delta_P(\widehat{\Xcal}^{\Rcal}_{1:M})}{\alpha}}+ c\sqrt{\frac{M\log\lceil\frac{n}{2}\rceil}{\alpha \lceil\frac{n}{2}\rceil}}.\end{multline*}
Marginalizing over the data points $\{(X_i,Y_i)\}_{i=1,\dots,\lfloor \frac{n}{2}\rfloor}$ used to choose the partition, 
we therefore have
\[\Pp{(X_i,Y_i)\iidsim P}{\pi_P(X_{n+1})\in \Ch^{\Rcal}_n(X_{n+1})} \geq 1- \alpha\]
and after applying Jensen's inequality,
\[\Ep{(X_i,Y_i)\iidsim P}{\leb(\Ch^{\Rcal}_n(X_{n+1}))} \leq L_\alpha(\Pi_P) + \sqrt{\frac{2\Ep{(X_i,Y_i)\iidsim P}{\Delta_P(\widehat{\Xcal}^{\Rcal}_{1:M})}}{\alpha}}+ c\sqrt{\frac{M\log\lceil\frac{n}{2}\rceil}{\alpha \lceil\frac{n}{2}\rceil}}.\]
Note that these bounds hold for any distribution $P$. The first statement therefore immediately verifies that 
$\Ch^{\Rcal}_n$
 satisfies the distribution-free coverage property~\eqref{eqn:distr_free_confidence_binary}.
We now complete the proof of the bound on length. 
Write $\textnormal{mid}_m = \frac{m(X) - \frac{1}{2}}{M}$, the midpoint of the range $[\frac{m-1}{M},\frac{m}{M})$ of estimated probabilities
that define the region $\widehat{\Xcal}^{\Rcal}_m$. We have
\begin{align*}
&\Delta\big(\widehat{\Xcal}^{\Rcal}_{1:M}\big) 
 = \Epst{P_X}{|\pi_P(X) - \pi_{P,m(X)}|}{\{(X_i,Y_i)\}_{i=1,\dots,\lfloor \frac{n}{2}\rfloor}}\\
&=\sum_{m=1}^M p_{P,m}  \Epst{P_X}{|\pi_P(X) - \pi_{P,m(X)}|}{X\in\widehat{\Xcal}^{\Rcal}_m;\{(X_i,Y_i)\}_{i=1,\dots,\lfloor \frac{n}{2}\rfloor}}\\
&\leq\sum_{m=1}^M p_{P,m}  \left(\left|\pi_{P,m} - \textnormal{mid}_m\right| + \Epst{P_X}{\left|\pi_P(X) - \textnormal{mid}_m\right|}{X\in\widehat{\Xcal}^{\Rcal}_m;\{(X_i,Y_i)\}_{i=1,\dots,\lfloor \frac{n}{2}\rfloor}}\right)\\
&\leq 2\sum_{m=1}^M p_{P,m}\Epst{P_X}{\left|\pi_P(X) - \textnormal{mid}_m\right|}{X\in\widehat{\Xcal}^{\Rcal}_m;\{(X_i,Y_i)\}_{i=1,\dots,\lfloor \frac{n}{2}\rfloor}}\\
&= 2\Epst{P_X}{\left|\pi_P(X) -\textnormal{mid}_{m(X)}\right|}{\{(X_i,Y_i)\}_{i=1,\dots,\lfloor \frac{n}{2}\rfloor}},
\end{align*}
where the next-to-last step holds since, for each $m$,
\begin{multline*}
\left|\pi_{P,m} - \textnormal{mid}_m\right|
= \left|  \Epst{P_X}{\pi_P(X)}{X\in\widehat{\Xcal}^{\Rcal}_m;\{(X_i,Y_i)\}_{i=1,\dots,\lfloor \frac{n}{2}\rfloor}} - \textnormal{mid}_m\right|\\
\leq  \Epst{P_X}{\left|\pi_P(X) - \textnormal{mid}_m\right|}{X\in\widehat{\Xcal}^{\Rcal}_m;\{(X_i,Y_i)\}_{i=1,\dots,\lfloor \frac{n}{2}\rfloor}},
\end{multline*}
by Jensen's inequality. Next, by definition of the partition $\widehat{\Xcal}^{\Rcal}_{1:M}$, we have
\[\left|\pi_P(X) - \textnormal{mid}_{m(X)}\right| \leq \left|\pi_P(X) - \widehat\pi^{\Rcal}_{\lfloor \frac{n}{2}\rfloor}(X)\right| + \frac{1}{2M}\]
almost surely. Therefore,
\begin{align*}
\EE{\Delta\big(\widehat{\Xcal}^{\Rcal}_{1:M}\big) }
&\leq \EE{\frac{1}{M} + 2\Epst{P_X}{\left|\pi_P(X) -  \widehat\pi^{\Rcal}_{\lfloor \frac{n}{2}\rfloor}(X)\right|}{\{(X_i,Y_i)\}_{i=1,\dots,\lfloor \frac{n}{2}\rfloor}}}\\
&=\frac{2}{M} + 2\Delta_{\lfloor\frac{n}{2}\rfloor,P}(\Rcal).
\end{align*}
Combining everything, we have
\[\Ep{(X_i,Y_i)\iidsim P}{\leb(\Ch^{\Rcal}_n(X_{n+1}))} \leq L_\alpha(\Pi_P) + \sqrt{\frac{2\left(\frac{1}{M} + 2\Delta_{\lfloor\frac{n}{2}\rfloor,P}(\Rcal)\right)}{\alpha}}+ c\sqrt{\frac{M\log\lceil\frac{n}{2}\rceil}{\alpha \lceil\frac{n}{2}\rceil}}.\]
Plugging in our choice of $M$, we have proved the desired bound when the universal constant $c'$ is chosen appropriately.

\subsection{Proofs of supporting lemmas}

\begin{proof}[Proof of Lemma~\ref{lem:ell_sqrt_bound}]
Without loss of generality we can take $t\leq\frac{1}{2}$. Furthermore, if $t'>\frac{1}{2}$ then $\ell(t',a+r\delta)\leq\ell(\frac{1}{2},a+r\delta)$,
so it suffices to consider only the case where $t'\leq\frac{1}{2}$ as well. Finally, if $t'<t$ then $\ell(t',a+r\delta)\leq \ell(t,a+r\delta)\leq \ell(t,a)$
since $\ell(t,a)$ is nondecreasing in $t\in[0,\frac{1}{2}]$
and nonincreasing in $a$,
so the result is trivial in this case. Therefore, from this point on we only need to prove the result for $0\leq t\leq t'\leq \frac{1}{2}$.

Now we split into cases.\begin{itemize}
\item If $t=a=0$, then $\ell(t,a) = \ell(0,0) = 0$ and
\[\ell(t',a+r\delta) \leq \frac{t'}{2(a+r\delta)} \leq \frac{2\delta^2}{2\cdot r\delta} = \frac{\delta}{r},\]
where the first step holds since $\ell(t,a)\leq \frac{t}{2a}$ for all $t,a$. 
\item If $a\geq \frac{1}{2}$, then
\begin{align*}\ell(t',a+\delta r) - \ell(t,a) - \frac{\delta}{r}
&= 2(1-a-\delta r)t' - 2(1-a)t - \frac{\delta}{r}\\
&= 2(1-a)(t'-t) - 2\delta r t' - \frac{\delta}{r}\\
&\leq 2(1-a)(t'-t) - \delta\sqrt{8t'} \textnormal{\quad since $2xy\leq x^2+y^2$ for all $x,y$}\\
&\leq t'-t - \delta\sqrt{8t'} \textnormal{\quad since $a\geq \frac{1}{2}$ implies $2(1-a)\leq 1$}\\
&\leq0,\end{align*}
where the last step holds by assumption on $|t'-t|$.
\item If $t\leq a <\frac{1}{2}$, then $\ell(t,a) = \frac{t}{2a}$, and we also know that $\ell(t',a+\delta r)\leq \frac{t'}{2(a+\delta r)}$ as before. Therefore,
\begin{align*}\ell(t',a+\delta r) - \ell(t,a) 
& \leq  \frac{t'}{2(a+\delta r)} -  \frac{t}{2a} \\
&\leq
\frac{t + \delta\sqrt{8t} + 2\delta^2}{2(a+\delta r)} - \frac{t}{2a}\textnormal{\quad by assumption on $|t'-t|$}\\
&\leq \frac{t + \delta r \cdot \frac{t}{a} + \frac{\delta}{r}\cdot 2a  + 2\delta^2}{2(a+\delta r)} - \frac{t}{2a}\textnormal{\quad since $2xy\leq x^2+y^2$ for all $x,y$}\\
&= \frac{ (a+\delta r)\cdot \left(\frac{t}{a} + \frac{2\delta}{r}\right)}{2(a+\delta r)} - \frac{t}{2a}= \frac{\delta}{r} .
\end{align*}
\item If $a < t$ and $a+\delta r<t'$, then
\begin{align*}\ell(t',a+\delta r) - \ell(t,a) - \frac{\delta}{r} 
&=
\left(1 - \frac{a+\delta r}{2t'}\right) - \left(1 - \frac{a}{2t}\right) - \frac{\delta}{r}\\
& = \frac{a\cdot (t'-t) }{2tt'} - \frac{\delta r + 2t'\delta/r}{2t'}\\
&\leq\frac{t'-t }{2t'} - \frac{\delta r + 2t'\delta/r}{2t'}\textnormal{\quad since $a<t$}\\
& \leq \frac{t'-t - \delta\sqrt{8t'}}{2t'}\textnormal{\quad since $2xy\leq x^2+y^2$ for all $x,y$}\\
&\leq 0,
 \end{align*}
where the last step holds by assumption on $|t'-t|$.
\item If $a<t$ and $a+\delta r\geq t'$, then $ \ell(t,a) =1 - \frac{a}{2t} \geq \frac{1}{2}$ and $\ell(t',a+\delta r)\leq \frac{t'}{2(a+\delta r)} \leq \frac{1}{2}$, and so
\[\ell(t',a+\delta r) - \ell(t,a) - \frac{\delta}{r} \leq 0.\]
\end{itemize}
We have now verified that the bound holds in all cases, which completes the proof.
\end{proof}

\begin{proof}[Proof of Lemma~\ref{lem:chernoff}]
Fix any $m$. We will show that the bounds hold for this choice of $m$ with probability at least $1-\frac{\alpha}{Mn}$,
and then the lemma follows by applying the union bound.

The multiplicative Chernoff bound~\citep[Theorem 2.3(b,c)]{mcdiarmid1998concentration} states that, for 
any integer $N\geq1$ and any $t\in[0,1]$, for a random variable $B\sim\textnormal{Binomial}(N,t)$, it holds for all $\delta>0$ that
\begin{equation}\label{eqn:chernoff1}
\PP{ B\geq Nt - \sqrt{2Nt\log(1/\delta)} }\geq 1-\delta\end{equation}
and
\begin{equation}\label{eqn:chernoff2}\PP{B \leq Nt + \sqrt{3Nt\log(1/\delta)}\vee 3\log(1/\delta)}\geq 1-\delta.\end{equation}

Now fix any $m$.
We will prove that the statements~\eqref{eqn:check_smoothing_p} and~\eqref{eqn:check_smoothing_pi} hold with probability at least $1-\frac{\alpha}{Mn}$ for this $m$.

First, we have $n\widehat p_m\sim\textnormal{Binomial}(n,p_{P,m})$, and therefore, 
applying~\eqref{eqn:chernoff1} with $N=n$, $t=p_{P,m}$, 
and $\delta= \frac{\alpha}{4Mn}$, with probability at least $1-\frac{\alpha}{4Mn}$ it holds that
\begin{equation}\label{eqn:p_hat_chernoff1}\widehat p_m \geq p_{P,m} - \sqrt{p_{P,m}\cdot \frac{2\log(4Mn/\alpha)}{n}}.\end{equation}
Furthermore,
applying~\eqref{eqn:chernoff2} with $N=n$, $t=p_{P,m}$, 
and $\delta= \frac{\alpha}{4Mn}$, with probability at least $1-\frac{\alpha}{4Mn}$ it holds that
\begin{equation}\label{eqn:p_hat_chernoff2}\widehat p_m \leq p_{P,m} + \sqrt{p_{P,m}\cdot \frac{3\log(4Mn/\alpha)}{n}}\vee \frac{3\log(4Mn/\alpha)}{n}.\end{equation}

Now assume that~\eqref{eqn:p_hat_chernoff1} and~\eqref{eqn:p_hat_chernoff2}
both hold. We will show that this implies~\eqref{eqn:check_smoothing_p} for this $m$.
First, by~\eqref{eqn:p_hat_chernoff2}, we have
\begin{multline*}
\tilde p_m 
=   \widehat p_m + \sqrt{\widehat p_m\cdot \frac{3\log(4Mn/\alpha)}{n}}+\frac{3\log(4Mn/\alpha)}{n}\\
\leq p_{P,m} + \sqrt{p_{P,m}\cdot \frac{18\log(4Mn/\alpha)}{n}} +\frac{12\log(4Mn/\alpha)}{n}.\end{multline*}
Next, by~\eqref{eqn:p_hat_chernoff1}, we have
\[
p_{P,m} \leq \widehat p_m + \sqrt{\widehat p_m\cdot \frac{2\log(4Mn/\alpha)}{n}}+\frac{2\log(4Mn/\alpha)}{n}
 \leq \tilde p_m - \frac{p_{P,m}}{n-1}.
\]
where the last step holds since $n\geq2$ and so $\frac{\log(4Mn/\alpha)}{n}\geq \frac{\log(4n)}{n} \geq \frac{1}{n-1}\geq \frac{p_{P,m}}{n-1}$.
Thus we have
\[p_{P,m}\leq \tilde p_m\left(1-\frac{1}{n}\right).\]
This verifies that~\eqref{eqn:check_smoothing_p} holds for this value of $m$.

Now we turn to~\eqref{eqn:check_smoothing_pi}.  We will condition on $\widehat p_m$.
We split into cases:
\begin{itemize}
\item First suppose $\pi_{P,m}$ satisfies
\[\pi_{P,m} + \sqrt{\pi_{P,m}\cdot \frac{3\log(4Mn/\alpha)}{n\widehat p_m}}\vee \frac{3\log(4Mn/\alpha)}{n\widehat p_m} \leq \frac{1}{2}.\] 
Then $\widehat p_m>0$, and conditional on $\widehat p_m$, we have $n\widehat p_m\widehat\pi_m\sim\textnormal{Binomial}(n\widehat p_m,\pi_{P,m})$.
Therefore, applying~\eqref{eqn:chernoff1} with $N=n\widehat p_m$, $t=\pi_{P,m}$, and $\delta=\frac{\alpha}{4Mn}$, 
with probability at least $1-\frac{\alpha}{4Mn}$ it holds that
\begin{equation}\label{eqn:pi_hat_chernoff1}\widehat\pi_m \geq \pi_{P,m} - \sqrt{\pi_{P,m}\cdot \frac{2\log(4Mn/\alpha)}{n\widehat p_m}}.\end{equation}
And,  applying~\eqref{eqn:chernoff2} with $N=n\widehat p_m$, $t=\pi_{P,m}$, and $\delta=\frac{\alpha}{4Mn}$, 
with probability at least $1-\frac{\alpha}{4Mn}$ it holds that
\begin{equation}\label{eqn:pi_hat_chernoff2}\widehat\pi_m \leq \pi_{P,m} + \sqrt{\pi_{P,m}\cdot \frac{3\log(4Mn/\alpha)}{n\widehat p_m}}\vee \frac{3\log(4Mn/\alpha)}{n\widehat p_m}.\end{equation}

Now assume that~\eqref{eqn:p_hat_chernoff1},~\eqref{eqn:pi_hat_chernoff1}, and~\eqref{eqn:pi_hat_chernoff2} all hold.
By definition of this case, the bound~\eqref{eqn:pi_hat_chernoff2} immediately implies that $\widehat\pi_m\leq \frac{1}{2}$ and therefore by definition, $\tilde\pi_m\leq \frac{1}{2}$ also.
Applying~\eqref{eqn:pi_hat_chernoff1},
\[\pi_{P,m}
\leq \widehat\pi_m+\sqrt{\widehat\pi_m\cdot \frac{2\log(4Mn/\alpha)}{n\widehat p_m}}+\frac{2\log(4Mn/\alpha)}{n\widehat p_m},\]
and $\pi_{P,m}\leq\frac{1}{2}$ by definition of this case, so
\[\tilde\pi_m= \min\left\{\frac{1}{2}, \widehat\pi_m+  \sqrt{\widehat\pi_m\cdot \frac{2\log(4Mn/\alpha)}{n\widehat p_m}} + \frac{2\log(4Mn/\alpha)}{n\widehat p_m}\right\}  \geq\pi_{P,m}.\]
By applying~\eqref{eqn:pi_hat_chernoff2}, we furthermore have
\begin{multline*}
\tilde\pi_m
\leq \widehat\pi_m+  \sqrt{\widehat\pi_m\cdot \frac{2\log(4Mn/\alpha)}{n\widehat p_m}} + \frac{2\log(4Mn/\alpha)}{n\widehat p_m}\\
\leq \pi_{P,m} +  \sqrt{\pi_{P,m}\cdot \frac{12\log(4Mn/\alpha)}{n\widehat p_m}} + \frac{8\log(4Mn/\alpha)}{n\widehat p_m}.\end{multline*}
 Next, if $\widehat p_m \geq \frac{2}{3}p_{P,m}$ then this last bound yields
\[\tilde\pi_m \leq  \pi_{P,m} +  \sqrt{\pi_{P,m}\cdot \frac{18\log(4Mn/\alpha)}{n p_{P,m}}} + \frac{12\log(4Mn/\alpha)}{n p_{P,m}},\]
while if instead $\widehat p_m < \frac{2}{3}p_{P,m}$ then by~\eqref{eqn:p_hat_chernoff1} we can see that $p_{P,m} \leq \frac{18\log(4Mn/\alpha)}{n}$, 
and so
\[\tilde\pi_m \leq \frac{1}{2} \leq \frac{9\log(4Mn/\alpha)}{n p_{P,m}} .\] Either way, then, we have verified that
 the statement~\eqref{eqn:check_smoothing_pi} holds, as desired.

\item Next,  suppose $\pi_{P,m}$ satisfies
\[\pi_{P,m}\leq \frac{1}{2} < \pi_{P,m} + \sqrt{\pi_{P,m}\cdot \frac{3\log(4Mn/\alpha)}{n\widehat p_m}}\vee \frac{3\log(4Mn/\alpha)}{n\widehat p_m} .\] 
If $\widehat p_m>0$, then conditional on $\widehat p_m$, we have $n\widehat p_m(1-\widehat\pi_m)\sim\textnormal{Binomial}(n\widehat p_m,1-\pi_{P,m})$.
Therefore, applying~\eqref{eqn:chernoff1} with $N=n\widehat p_m$, $t=1-\pi_{P,m}$, and $\delta=\frac{\alpha}{4Mn}$, 
with probability at least $1-\frac{\alpha}{4Mn}$ it holds that
\begin{equation}\label{eqn:pi_hat_chernoff2_}(1-\widehat\pi_m) \geq (1-\pi_{P,m}) - \sqrt{(1-\pi_{P,m})\cdot \frac{2\log(4Mn/\alpha)}{n\widehat p_m}}.\end{equation}
Furthermore the bound~\eqref{eqn:pi_hat_chernoff1} holds with probability at least $1-\frac{\alpha}{4Mn}$ as above.
(If instead $\widehat p_m=0$, then the bounds~\eqref{eqn:pi_hat_chernoff1} and~\eqref{eqn:pi_hat_chernoff2_} hold trivially.)

Now assume that~\eqref{eqn:p_hat_chernoff1},~\eqref{eqn:pi_hat_chernoff1},  and~\eqref{eqn:pi_hat_chernoff2_} all hold.
If $\widehat\pi_m\leq \frac{1}{2}$, then by definition we have $\tilde\pi_m\leq \frac{1}{2}$ also.
Furthermore, applying~\eqref{eqn:pi_hat_chernoff1} proves that  $\pi_{P,m}\leq \tilde\pi_m$ exactly as for the previous case.
If instead $\widehat\pi_m> \frac{1}{2}$, then~\eqref{eqn:pi_hat_chernoff2_}
implies that
\[ \frac{1}{2}\leq (1-\pi_{P,m}) \leq (1-\widehat\pi_m) + \sqrt{(1-\widehat\pi_m)\cdot \frac{2\log(4Mn/\alpha)}{n\widehat p_m}} +  \frac{2\log(4Mn/\alpha)}{n\widehat p_m},\]
which therefore means that
\[\tilde\pi_m
=\max\left\{\frac{1}{2}, \widehat\pi_m-  \sqrt{(1-\widehat\pi_m)\cdot \frac{2\log(4Mn/\alpha)}{n\widehat p_m}} - \frac{2\log(4Mn/\alpha)}{n\widehat p_m}\right\} = \frac{1}{2}.\]
We have therefore established that $\pi_{P,m}\leq \tilde\pi_m\leq \frac{1}{2}$ under either scenario.
Next, 
\[|\tilde\pi_m - \pi_{P,m}| \leq \frac{1}{2} - \pi_{P,m} < \sqrt{\pi_{P,m}\cdot \frac{3\log(4Mn/\alpha)}{n\widehat p_m}}\vee \frac{3\log(4Mn/\alpha)}{n\widehat p_m} \]
by definition of this case.
As for the previous case, we can therefore verify~\eqref{eqn:check_smoothing_pi} by considering the two possibilities
 $\widehat p_m \geq \frac{2}{3}p_{P,m}$ and  $\widehat p_m < \frac{2}{3}p_{P,m}$, and applying~\eqref{eqn:p_hat_chernoff1}.

\item The case that $\pi_{P,m}>\frac{1}{2}$ is treated analogously.
\end{itemize}

\noindent We have now verified that, for each $m$, with probability at least $1-\frac{\alpha}{Mn}$, the bounds~\eqref{eqn:check_smoothing_p} and~\eqref{eqn:check_smoothing_pi} 
both hold, which completes the proof of the lemma.
\end{proof}
\end{document}